%% file: Imprimitivity.tex
\documentclass[11pt]{article}
\usepackage{amsmath,amssymb,amsthm,amscd}

\newtheorem{theorem}{Theorem}[section]
\newtheorem{lemma}[theorem]{Lemma}

\newtheorem{corollary}[theorem]{Corollary}

\theoremstyle{plain}

\theoremstyle{definition}
\newtheorem{definition}[theorem]{Definition}
\newtheorem{remark}[theorem]{Remark}

\numberwithin{equation}{section}

\renewcommand{\labelenumi}{\textup{(\theenumi)}}

\renewcommand{\phi}{\varphi}

\newcommand{\Homeo}{\operatorname{Homeo}}

\newcommand{\id}{\operatorname{id}}
\newcommand{\Ker}{\operatorname{Ker}}

\newcommand{\Ad}{\operatorname{Ad}}

\newcommand{\K}{\mathcal{K}}
\newcommand{\A}{\mathcal{A}}
\newcommand{\B}{\mathcal{B}}
\newcommand{\E}{\mathcal{E}}

\newcommand{\Z}{\mathbb{Z}}

\def\OA{{{\mathcal{O}}_A}}
\def\DA{{{\mathcal{D}}_A}}
\def\OB{{{\mathcal{O}}_B}}
\def\DB{{{\mathcal{D}}_B}}

\title{Imprimitivity bimodules of Cuntz--Krieger algebras \\
and strong shift equivalences of matrices}
\author{Kengo Matsumoto \\
Department of Mathematics \\
Joetsu University of Education \\
Joetsu, 943-8512, Japan
}
\date{ }

\begin{document}
\maketitle

\def\det{{{\operatorname{det}}}}

\begin{abstract}
In this paper, we will introduce a notion of 
basis related Morita equivalence 
in the Cuntz--Krieger algebras 
$(\OA, \{S_a\}_{a \in E_A})$
with the canonical right finite basis
$\{S_a\}_{a \in E_A}$
as Hilbert $C^*$-bimodule,
and prove that 
two nonnegative irreducible matrices 
$A$ and $B$ are elementary equivalent, that is,
$A = CD, B = DC$ for some nonnegative rectangular matrices $C, D$,
if and only if
the Cuntz--Krieger algebras  
$(\OA, \{S_a\}_{a \in E_A})$ and $(\OB, \{ S_b\}_{b\in E_B})$ 
with the canonical right finite bases 
are basis relatedly elementary Morita equivalent.
\end{abstract}




\def\OA{{{\mathcal{O}}_A}}
\def\OB{{{\mathcal{O}}_B}}
\def\OZ{{{\mathcal{O}}_Z}}
\def\OTA{{{\mathcal{O}}_{\tilde{A}}}}
\def\SOA{{{\mathcal{O}}_A}\otimes{\mathcal{K}}}
\def\SOB{{{\mathcal{O}}_B}\otimes{\mathcal{K}}}
\def\SOZ{{{\mathcal{O}}_Z}\otimes{\mathcal{K}}}
\def\SOTA{{{\mathcal{O}}_{\tilde{A}}\otimes{\mathcal{K}}}}
\def\DA{{{\mathcal{D}}_A}}
\def\DB{{{\mathcal{D}}_B}}
\def\DZ{{{\mathcal{D}}_Z}}
\def\DTA{{{\mathcal{D}}_{\tilde{A}}}}
\def\SDA{{{\mathcal{D}}_A}\otimes{\mathcal{C}}}
\def\SDB{{{\mathcal{D}}_B}\otimes{\mathcal{C}}}
\def\SDZ{{{\mathcal{D}}_Z}\otimes{\mathcal{C}}}
\def\SDTA{{{\mathcal{D}}_{\tilde{A}}\otimes{\mathcal{C}}}}
\def\BC{{{\mathcal{B}}_C}}
\def\BD{{{\mathcal{B}}_D}}
\def\OAG{{\mathcal{O}}_{A^G}}
\def\OBG{{\mathcal{O}}_{B^G}}
\def\Max{{{\operatorname{Max}}}}
\def\Per{{{\operatorname{Per}}}}
\def\PerB{{{\operatorname{PerB}}}}
\def\Homeo{{{\operatorname{Homeo}}}}
\def\HA{{{\frak H}_A}}
\def\HB{{{\frak H}_B}}
\def\HSA{{H_{\sigma_A}(X_A)}}
\def\Out{{{\operatorname{Out}}}}
\def\Aut{{{\operatorname{Aut}}}}
\def\Ad{{{\operatorname{Ad}}}}
\def\Inn{{{\operatorname{Inn}}}}
\def\det{{{\operatorname{det}}}}
\def\exp{{{\operatorname{exp}}}}
\def\cobdy{{{\operatorname{cobdy}}}}
\def\Ker{{{\operatorname{Ker}}}}
\def\ind{{{\operatorname{ind}}}}
\def\id{{{\operatorname{id}}}}
\def\supp{{{\operatorname{supp}}}}
\def\co{{{\operatorname{co}}}}
\def\Sco{{{\operatorname{Sco}}}}
\def\ActA{{{\operatorname{Act}_{\DA}(\mathbb{T},\OA)}}}
\def\ActB{{{\operatorname{Act}_{\DB}(\mathbb{T},\OB)}}}
\def\U{{{\mathcal{U}}}}

\section{Introduction}

Let
$A=[A(i,j)]_{i,j=1}^N$
be an irreducible matrix  
with entries in nonnegative integers,
that is simply called a nonnegative matrix.
We assume that $A$ is not any permutation matrix.
Let $G_A =(V_A, E_A)$ 
be a finite directed graph with $N$ vertices $ V_A =\{ v^A_1,\dots,v^A_N\}$
and with 
$A(i,j)$ directed edges whose source vertices
are $v_i$ and terminal vertices are  $v_j$.
For a directed edge $e$,
we denote 
by $s(e)$ the source vertex of $e$ and 
by $t(e)$ the terminal vertex of $e$.
Let $E_{A} =\{ a_1, \dots, a_{N_A}\}$
be the edge set of the graph $G_A$. 
 The two-sided topological Markov shift 
$(\bar{X}_A,\bar{\sigma}_A)$ 
associated with 
$A$ 
is defined as a topological dynamical system of the compact 
Hausdorff space
$
\bar{X}_A
$
of all biinfinte sequences 
$
(a_i)_{i \in \Z}
\in
E_A^{\Z}
 $
 of edges $a_i$ of $G_A$ 
 such that 
$  t(a_i) = s(a_{i+1}) $
for all 
$i \in \Z $
with shift transformation
$\bar{\sigma}_A$ defined by
$
\bar{\sigma}_A( (a_i)_{i\in \Z}) =(a_{i+1})_{i\in \Z}. 
$
R. F. Williams in \cite{Williams} proved 
a fundamental classification result for the topological Markov shifts
which says 
that the topological Markov shifts 
$(\bar{X}_A,\bar{\sigma}_A)$ and 
$(\bar{X}_B,\bar{\sigma}_B)$ are topologically conjugate
if and only if the matrices $A$ and $B$ are strong shift equivalent.
Two nonnegative square matrices 
$M$ and $N$ are said to be  elementary equivalent, or strong shift equivalent in $1$-step,
if there exist
nonnegative rectangular matrices $R,S$ such that 
$M = RS, N = SR$.
If there exists a finite sequence of nonnegative matrices $A_1,A_2,\dots, A_k$
such that
$A = A_1, B = A_k$ and $A_i$ are elementary  equivalent to $A_{i+1}$ for $i=1,2,\dots, k-1$, 
then $A$ and $B$ are said to be strong shift equivalent.

In \cite{CK}, Cuntz--Krieger introduced and studied a class of $C^*$-algebras associated to topological Markov shifts. 
They are well-known and called the Cuntz--Krieger algebras.  
There is a standard method to associate a Cuntz--Krieger algebra from a square matrix with entries in nonnegative integers as described in \cite[Section 4]{Ro}. 
For a nonnegative matrix  $A =[A(i,j)]_{i,j=1}^N $,
 the associated directed graph  
$G_A = (V_A,E_A)$
with the edge set 
$E_A =\{ a_1, \dots, a_{N^A}\}$ 
 has the $N_A \times N_A$
 transition matrix $A^G =[A^G(a_i,a_j)]_{i,j=1}^{N_A}$ of edges 
defined by 
\begin{equation}
A^G(a_i,a_j) =
\begin{cases} 
 1 &  \text{  if  } t(a_i) = s(a_j), \\
 0 & \text{  otherwise}
\end{cases}\label{eq:AG}
\end{equation}
for $a_i, a_j \in E_A$.
The Cuntz--Krieger algebra $\OA$ for the matrix $A$ with
 entries in nonnegative integers  is defined as the Cuntz--Krieger algebra
 ${\mathcal{O}}_{A^G}$ 
for the matrix $A^G$ which is the universal $C^*$-algebra generated by 
partial isometries
$S_{a_i}$ indexed by edges $a_i, i=1,\dots, N_A$ subject to the relations:
\begin{equation}
\sum_{j=1}^{N_A}  S_{a_j} S_{a_j}^* = 1, 
\qquad
S_{a_i}^* S_{a_i} = 
\sum_{j=1}^{N_A}  A^G(a_i, a_j ) S_{a_j} S_{a_j}^*
 \quad \text{ for } i=1,\dots,N_A.
\label{eq:OAG}
\end{equation}
Since we are assuming that the matrix $A$ is irreducible and not any permutation,
the algebra $\OA$ is uniquely determined by the relation
\eqref{eq:OAG}, and becomes simple and purely infinite
(\cite{CK}).
For a word 
$\mu =(\mu_1, \dots, \mu_k), \mu_i  \in E_A$,
we denote by $S_\mu$ the partial isometry
$S_{\mu_1}\cdots S_{\mu_k}$.  
For $t \in {\mathbb{R}}/\Z = {\mathbb{T}}$,
the correspondence
$S_{a_i} \rightarrow e^{2 \pi\sqrt{-1}t}S_{a_i},
\, i=1,\dots,N_A$
gives rise to an automorphism
of $\OA$ denoted by
$\rho^A_t$.
The automorphisms
$\rho^A_t, t \in {\mathbb{T}}$
yield an action of ${\mathbb{T}}$
on $\OA$ called the gauge action.
 Let  us denote by 
$\DA$ 
the $C^*$-subalgebra of $\OA$
generated by the projections of the form:
$S_{a_{i_1}}\cdots S_{a_{i_n}}S_{a_{i_n}}^* \cdots S_{a_{i_1}}^*,
i_1,\dots,i_n =1,\dots,N_A$.
Let $\K$ be the $C^*$-algebra of all compact operators on a separable infinite dimensional Hilbert space.
 Cuntz and Krieger proved that
 if two topological Markov shifts
 $(\bar{X}_A,\bar{\sigma}_A)$ and $(\bar{X}_B,\bar{\sigma}_B)$
  are topologically conjugate, 
the gauge actions $\rho^A$ and $\rho^B$
of the Cuntz-Krieger algebras
 $\OA$ and $\OB$ 
 are stably conjugate.
 That is,
 there exists an isomorphism $\phi$ from 
 $\OA \otimes \K$ to
 $\OB \otimes \K$
 such that 
 $\phi \circ (\rho_t^A \otimes \id) = (\rho_t^B \otimes \id) \circ \phi,
\, t \in {\mathbb{T}}$
 (\cite[3.8. Theorem]{CK}, cf. \cite{Cu3}).
As a corollary of this result,
 we know that 
if two nonnegative irreducible matrices 
$A$ and $B$ are strong shift equivalent, 
the stabilized Cuntz--Krieger algebras 
$\OA \otimes \K$ and 
 $\OB \otimes \K$
 are isomorphic.
On the other hand,
M. Rieffel in \cite{Rieffel1} has introduced the notion of 
strong Morita equivalence in $C^*$-algebras
from the view point of representation theory.
Two $C^*$-algebras $\A$ and $\B$ are said to be 
strongly Morita equivalent if there exists an $\A$--$\B$-imprimitivity bimodule.
By Brown--Green--Rieffel theorem \cite[Theorem 1.2]{BGR},
 two unital $C^*$-algebras 
$\A$ and $\B$
are strongly Morita equivalent 
if and only if their stabilizations 
$\A\otimes\K$ and $\B\otimes\K$ 
are isomorphic
(cf. \cite{Brown}, \cite{BGR}, \cite{Combes}). 
From this view point, 
the author has recently introduced a notion of 
strongly  Morita equivalent for
 the triplets 
$(\OA, \DA, \rho^A)$
called the Cuntz--Krieger triplet,
and proved that 
$(\OA, \DA, \rho^A)$
and 
$(\OB,\DB,\rho^B)$
are strongly Morita equivalent
if and only if
$A$ and $B$ are strong shift equivalent
(\cite{MaPre2016}).
We note that Morita equivalence of $C^*$-algebras 
from view points of strong shift equivalence of matrices
had been studied in 
\cite{MaETDS2004},
\cite{MaYMJ2007},
\cite{MPT},
\cite{Tomforde},
etc.

Y. Watatani in \cite{Watatani}
has introduced the notion of finite basis of  Hilbert $C^*$-module.
In this paper, inspired by his definition of finite basis, we 
 we will introduce a notion of 
{\it basis relatedly Morita equivalence}\/
in the Cuntz--Krieger algebras 
$(\OA, \{S_a\}_{a \in E_A})$
with the canonical right finite basis
$\{S_a\}_{a \in E_A}$
of the generating partial isometries satisfying \eqref{eq:OAG}
as Hilbert $C^*$-bimodule.
For nonnegative irreducible matrices $A$ and $B$,
the Cuntz--Krieger algebras
$(\OA, \{S_a\}_{a \in E_A})$
and
$(\OB, \{S_b\}_{b \in E_B})$
with their respect  canonical right  finite bases
$\{S_a\}_{a \in E_A}$
and
$\{S_b\}_{b \in E_B}$
are said to be {\it basis relatedly elementary Morita equivalent}\/
if there exists an $\OA$--$\OB$-implimitivity bimodule
${}_A\!X_B$ with 
a right finite basis $\{\eta_c\}_{c \in \E_C}$
and
a left finite basis $\{\zeta_d\}_{d \in \E_D}$
that are finitely related such that the relative tensor products
between 
${}_A\!X_B$ and its conjugate
$\overline{{}_A\!X_B}$ are isomorphic to $\OA$ and $\OB$, respectively:
\begin{equation}
\OA \cong {}_A\!X_B\otimes_{\OB}\overline{{}_A\!X_B}, \qquad
\OB \cong \overline{{}_A\!X_B}\otimes_{\OA}{}_A\!X_B
\end{equation}
as Hilbert $C^*$-bimodule with respect to their canonical right finite bases. 
If 
$(\OA, \{S_a\}_{a \in E_A})$
and
$(\OB, \{S_b\}_{b \in E_B})$
are connected by a finite sequence  of 
basis related elementary Morita equivalences,
then they are said to be 
{\it basis relatedly Morita equivalent}.
We will prove the following theorem.  
\begin{theorem}[Theorem \ref{thm:main3}]
Let $A, B$ be nonnegative irreducible matrices that are not any permutations.
Then 
$A$ and $B$ are elementary equivalent, that is,
$A = CD, B = DC$ for some nonnegative rectangular matrices $C, D$,
if and only if
the Cuntz--Krieger algebras  
$(\OA, \{S_a\}_{a \in E_A})$ and $(\OB, \{ S_b\}_{b\in E_B})$ 
with the canonical right finite bases
are basis relatedly elementary Morita equivalent.
\end{theorem}
Thanks to the Williams' classification theorem,
we have the following corollary. 
\begin{corollary}[Corollary \ref{cor:main4}, \cite{MaPre2016}]
Let $A, B$ be nonnegative irreducible matrices that are not any permutations.
Then the following three assertions are equivalent.
\begin{enumerate}
\renewcommand{\theenumi}{\roman{enumi}}
\renewcommand{\labelenumi}{\textup{(\theenumi)}}
\item 
The two-sided topological Markov shifts 
$(\bar{X}_A, \bar{\sigma}_A)$ and  
$(\bar{X}_B, \bar{\sigma}_B)$
are topologically conjugate.
\item
The Cuntz--Krieger algebras  
$(\OA, \{S_a\}_{a \in E_A})$ and $(\OB, \{ S_b\}_{b\in E_B})$ 
with the canonical right finite bases are basis relatedly Morita equivalent.
\item The Cuntz--Krieger triplets
$(\OA, \DA, \rho^A)$
and 
$(\OB,\DB,\rho^B)$
are strongly Morita equivalent
in the sense of \cite{MaPre2016}.
\end{enumerate}
\end{corollary}


\section{ Basis related Morita equivalence}
Let $\A$ and $\B$ be $C^*$-algebras.
Let us first recall the definition of Hilbert $C^*$-module
introduced by Paschke \cite{Paschke} (cf. \cite{Rieffel1}, \cite{Kasparov}, \cite{KW}, etc.).
A left Hilbert $C^*$-module ${}_\A\!X$ over $\A$ is a $\mathbb{C}$-vector space 
with a left $\A$-module structure and an $\A$-valued inner product 
${}_\A\!\langle\hspace{3mm}\mid\hspace{3mm}\rangle$
satisfying the following conditions \cite[Definition 1.1]{KW}:
\begin{enumerate}
\item ${}_\A\!\langle\hspace{3mm}\mid\hspace{3mm}\rangle$ 
is left linear  and right conjugate linear.
\item ${}_\A\!\langle a x\mid y \rangle = a {}_\A\!\langle x\mid y \rangle$
and ${}_\A\!\langle x \mid a y \rangle = {}_\A\!\langle x \mid y \rangle a^* $ 
for all
$x,y \in {}_\A\!X $ and $a \in \A$.
\item ${}_\A\!\langle x \mid x \rangle \ge 0$ for all
$x\in {}_\A\!X$, and ${}_\A\!\langle x \mid x \rangle = 0$ 
if and only if $x =0$.
\item ${}_\A\!\langle x \mid y  \rangle = {}_\A\!\langle y \mid x \rangle^*$ for all
$x,y \in {}_\A\!X$.
\item ${}_\A\!X$ is complete with respect to the norm 
$\| x \| = \| {}_\A\!\langle x \mid x \rangle \|^{\frac{1}{2}}$.
\end{enumerate}
It is said to be (left) full if the closed linear span of
$\{ {}_\A\!\langle x \mid  y \rangle \mid x,y \in {}_\A\!X \}$ is equal to $\A$.

Similarly a right Hilbert $C^*$-module $X_\B$ over $\B$ 
is defined as a $\mathbb{C}$-vector space 
with a right $\B$-module structure and a $\B$-valued inner product 
$\langle\hspace{3mm}\mid\hspace{3mm}\rangle_\B$
satisfying the following conditions \cite[Definition 1.2]{KW}:
\begin{enumerate}
\item $\langle\hspace{3mm}\mid\hspace{3mm}\rangle_\B$ 
is left conjugate and right linear.
\item $\langle x\mid y b \rangle_\B = \langle x\mid y \rangle_\B b$
and $\langle x b\mid y \rangle_\B = b^*\langle x \mid y \rangle_\B $ 
for all
$x,y \in X_\B $ and $b \in \B$.
\item $\langle x \mid x \rangle_\B \ge 0$ for all
$x\in X_\B$, and $\langle x \mid x \rangle_\B = 0$ 
if and only if $x =0$.
\item $\langle x \mid y  \rangle_\B = \langle y \mid x \rangle_\B^*$ for all
$x,y \in X_\B$.
\item $X_\B$ is complete with respect to the norm 
$\| x \| = \| \langle x \mid x \rangle_\B \|^{\frac{1}{2}}$.
\end{enumerate}
It is said to be (right) full if the closed linear span of
$\{ \langle x \mid  y \rangle_\B \mid x,y \in X_\B \}$ is equal to $\B$.

Throughout the paper,
we mean by an $\A$--$\B$-bimodule
written ${}_{\A}\!X_\B$ 
a left Hilbert $C^*$-module over $\A$
and also 
a right Hilbert $C^*$-module over $\B$
in the above sense (\cite{Rieffel1}, \cite{KW},  \cite{RW}, etc. ).
In \cite[Definition 6.10]{Rieffel1}, M. Rieffel has defined 
the notion of an $\A$--$\B$-imprimitivity bimodule
in the following way.
An $\A$--$\B$-bimodule ${}_{\A}\!X_\B$ is said to be 
$\A$--$\B$-imprimitivity bimodule
if the three conditions below hold
\begin{enumerate}
\item ${}_{\A}\!X_\B$ is a full left Hilbert $\A$-module
with $\A$-valued left inner product 
${}_\A\!\langle\hspace{3mm}\mid\hspace{3mm}\rangle$,
 and a full right Hilbert $\B$-module
with $\B$-valued right inner product 
$\langle\hspace{3mm}\mid\hspace{3mm}\rangle_\B$,
\item $\langle a\cdot x \mid y \rangle_{\B} =\langle  x \mid a^* \cdot y \rangle_{\B}$ and 
 ${}_{\A}\!\langle  x\cdot b \mid y \rangle ={}_{\A}\!\langle  x \mid  y\cdot b^* \rangle$
for all $x, y \in {}_{\A}\!X_\B$ and $a \in \A, \, b \in \B$.
\item
${}_{\A}\!\langle  x \mid y \rangle\cdot z = x \cdot \langle  y \mid  z \rangle_{\B}$
for all $x, y, z \in {}_{\A}\!X_\B$.
\end{enumerate}
We note that the above condition (2) implies 
\begin{equation}
\| {}_\A\!\langle x\mid x \rangle \| 
= 
\| \langle x\mid x \rangle_\B \| 
\qquad x \in {}_\A\!X_\B \label{eq:norm}
\end{equation}
so that the two norms on  
${}_\A\!X_\B$ induced by 
the left hand side and the right hand side of \eqref{eq:norm} coincide 
(cf. \cite[Corollary 1.19]{KW}, \cite[Proposition 3.11]{RW}).

Y. Watatani in \cite{Watatani}
has introduced the notion of a finite basis of a Hilbert $C^*$-module to construct $C^*$-index theory.
We follow his definition of finite basis.
An $\A$--$\B$-imprimitivity bimodule ${}_{\A}\!X_\B$
has a right finite  basis if there exists a finite family
$\{ \eta_c\}_{c \in \E_C} \subset {}_{\A}\!X_\B$
of vectors
indexed by a finite set $\E_C$
such that 
\begin{equation}
x = \sum_{c \in \E_C}\eta_c \langle \eta_c \mid x \rangle_{\B} 
\quad \text{ for } x \in {}_{\A}\!X_\B. \label{eq:basis}
\end{equation}
We note that the condition \eqref{eq:basis}
is equivalent to the condition
\begin{equation}
 \sum_{c \in \E_C} {}_\A\!\langle \eta_c \mid \eta_c \rangle =1 
\qquad \text{ in } \A \label{eq:remarkbasis}
\end{equation}
if $\A$ is unital and the center 
$Z(\A)$ of $\A$ is $\mathbb{C}$,
because of the identity
$\eta_c \langle \eta_c \mid x \rangle_{\B} 
={}_\A\!\langle \eta_c \mid \eta_c \rangle x$
(\cite[Proposition 1.13]{KW}). 
Similarly 
an $\A$--$\B$-imprimitivity bimodule ${}_{\A}\!X_\B$
has a left finite  basis if there exists a finite family
$\{ \zeta_d\}_{d \in \E_D} \subset {}_{\A}\!X_\B$
of vectors
indexed by a finite set $\E_D$
such that 
\begin{equation}
x = \sum_{d \in \E_D}  {}_\A\!\langle x\mid \zeta_d  \rangle \zeta_d 
\quad \text{ for } x \in {}_{\A}\!X_\B. \label{eq:lbasis}
\end{equation}

\begin{definition}
Let ${}_{\A}\!X_\B$ be an 
$\A$--$\B$-imprimitivity bimodule 
with a right finite  basis $\{ \eta_c\}_{c \in \E_C} \subset {}_{\A}\!X_\B$
and
a left finite  basis $\{ \zeta_d\}_{d \in \E_D} \subset {}_{\A}\!X_\B$,
respectively.
The bases $\{ \eta_c\}_{c \in \E_C}$ and  
$\{ \zeta_d\}_{d \in \E_D}$ 
are said to be {\it finitely related}
if there exist
\begin{align*}
\text{an } \E_C\times \E_D-\text{matrix }
\tilde{C} & = [\tilde{C}(c,d)]_{c \in \E_C, d \in \E_D} \text{ with entries in } \{0,1\},
\text{ and } \\
\text{an } \E_D\times \E_C-\text{matrix }
\tilde{D} & = [\tilde{D}(d,c)]_{d \in \E_D, c \in \E_C} \text{ with entries in } \{0,1\}  
\end{align*}
such that 
\begin{align}
\langle \eta_c \mid \eta_c \rangle_\B 
& = \sum_{d \in \E_D}\tilde{C}(c,d) 
\langle \zeta_d \mid \zeta_d \rangle_\B, \qquad c \in \E_C, \label{eq:3.1eta}\\
{}_\A\!\langle \zeta_d \mid \zeta_d \rangle 
& = \sum_{c \in \E_C}\tilde{D}(d,c) 
{}_\A\!\langle \eta_c \mid \eta_c\rangle, \qquad d \in \E_D  \label{eq:3.1zeta}\\
\intertext{ and }
{}_\A\!\langle \zeta_d \mid \eta_c \rangle 
\cdot {}_\A\!\langle  \eta_c \mid \zeta_d \rangle 
& = \tilde{C}(c,d) 
{}_\A\!\langle \zeta_d \mid \zeta_d \rangle, \qquad c \in \E_C, \, d \in \E_D,
 \label{eq:3.1zetaeta}\\
\langle \eta_c \mid \zeta_d \rangle_\B 
\cdot\langle \zeta_d \mid \eta_c  \rangle_\B 
& = \tilde{D}(d,c) 
\langle \eta_c \mid \eta_c \rangle_\B, \qquad c \in \E_C, \, d \in \E_D. \label{eq:3.1etazeta}
\end{align}
In this case the quintuplet
$({}_{\A}\!X_\B, \{ \eta_c\}_{c \in \E_C}, \{ \zeta_d\}_{d \in \E_D}, \tilde{C}, \tilde{D})$
is called a {\it bipartitely related}\/ 
$\A$--$\B$-imprimitivity bimodule.
\end{definition}

\begin{definition}
Let ${}_{\A}\!X_\B$ be an 
$\A$--$\B$-imprimitivity bimodule 
with right finite  basis $\{ \eta_c\}_{c \in \E_C} \subset {}_{\A}\!X_\B$,
and
${}_{\B}\!Y_\A$ be a 
$\B$--$\A$-imprimitivity bimodule 
with right finite  basis $\{ \xi_d\}_{d \in \E_D} \subset {}_{\B}\!Y_\A$,
respectively.
The bases $\{ \eta_c\}_{c \in \E_C}$ and  
$\{ \xi_d\}_{d \in \E_D}$ 
are said to be {\it finitely related}
if there exist
\begin{align*}
\text{an } \E_C\times \E_D-\text{matrix }
\tilde{C} & = [\tilde{C}(c,d)]_{c \in \E_C, d \in \E_D} \text{ with entries in } \{0,1\},
\text{ and } \\
\text{an } \E_D\times \E_C-\text{matrix }
\tilde{D} & = [\tilde{D}(d,c)]_{d \in \E_D, c \in \E_C} \text{ with entries in } \{0,1\}  
\end{align*}
such that 
\begin{align}
\langle \eta_c \mid \eta_c \rangle_\B & = \sum_{d \in \E_D}\tilde{C}(c,d) 
{}_\B\!\langle \xi_d \mid \xi_d \rangle, \qquad c \in \E_C, \label{eq:3.2eta}\\
\langle \xi_d \mid \xi_d \rangle_\A & = \sum_{c \in \E_C}\tilde{D}(d,c) 
{}_\A\!\langle \eta_c \mid \eta_c\rangle, \qquad d \in \E_D  \label{eq:3.2xi}\\
\intertext{ and }
\langle \langle \eta_c \mid \eta_c \rangle_\B \xi_d
\mid \xi_d \rangle_\A 
& = \tilde{C}(c,d) 
\langle \xi_d \mid \xi_d \rangle_\A, \qquad c \in \E_C, \, d \in \E_D \label{eq:3.2etaxi}\\
\langle \langle \xi_d \mid \xi_d \rangle_\A \eta_c 
\mid \eta_c \rangle_\B 
& = \tilde{D}(d,c) 
\!\langle \eta_c \mid \eta_c \rangle_\B, \qquad c \in \E_C, \, d \in \E_D. \label{eq:3.2xieta}
\end{align}
We will know from \eqref{eq:3.tensorA}
that the left hand sides of  \eqref{eq:3.2etaxi} and \eqref{eq:3.2xieta}
will coincide with  the inner products 
$\langle \eta_c\otimes_\B \xi_d \mid \eta_c\otimes_\B \xi_d \rangle_\A$
and
$\langle \xi_d \otimes_\A \eta_c \mid \xi_d\otimes_\A \eta_c \rangle_\B$
of their respect  relative tensor products. 

If two imprimitivity bimodules 
${}_{\A}\!X_\B$ and
${}_{\B}\!Y_\A$
have finitely related  basis, 
they are called 
{\it finitely related imprimitivity bimodules}.

For an $\A$--$\B$ imprimitivity bimodule  
${}_\A\!X_{\B}$,
its conjugate module
$\overline{{}_\A\!X_{\B}}$
as a $\B$--$\A$ imprimitivity bimodule
is defined in the following way (\cite[Definition 6.17]{Rieffel1}, cf. \cite{KW}, \cite{RW}).
The module $\overline{{}_\A\!X_{\B}}$
is
${}_\A\!X_{\B}$ itself as a set.
Let us denote by $\bar{x}$ in $\overline{{}_\A\!X_{\B}}$
the element $x $ in ${{}_\A\!X_{\B}}$. 
Its module structure with inner products are defined by 
\begin{enumerate}
\renewcommand{\theenumi}{\roman{enumi}}
\renewcommand{\labelenumi}{\textup{(\theenumi)}}
\item 
$\bar{x} + \bar{y} := \overline{x+y}$ and 
$\lambda\cdot \bar{x}:= \overline{ \bar{\lambda} x}$ 
for all $\lambda \in {\mathbb{C}}, x,y \in {}_\A\!X_{\B}.$
\item
$b\cdot \bar{x} \cdot a :=  \overline{a^* x b^*}$
for all $ a\in \A, b \in \B, x \in {}_\A\!X_{\B}.$
\item
${}_\B\!\langle \bar{x}, \bar{y}\rangle:= \langle {x}, {y}\rangle_\B$
and
$\langle \bar{x}, \bar{y}\rangle_\A:= {}_\A\!\langle {x}, {y}\rangle$
for all $x,y \in {}_\A\!X_{\B}.$
\end{enumerate}
The following lemma is straightforward.
\begin{lemma}
Let ${}_{\A}\!X_\B$ be an 
$\A$--$\B$-imprimitivity bimodule 
with right finite  basis $\{ \eta_c\}_{c \in \E_C} \subset {}_{\A}\!X_\B$
and
left finite  basis $\{ \zeta_d\}_{d \in \E_D} \subset {}_{\A}\!X_\B$.
Then 
$\{ \eta_c\}_{c \in \E_C}$
and
$\{ \zeta_d\}_{d \in \E_D}
$
are finitely related in ${}_{\A}\!X_\B$
if and only if
$\{ \eta_c\}_{c \in \E_C}$ in ${}_{\A}\!X_\B$
and
the conjugate basis $\{ \bar{\zeta}_d\}_{d \in \E_D}
$
in 
$\overline{{}_{\A}\!X_\B}$
are finitely related.
\end{lemma}

For $\A=\B$, a right finite basis
$\{ s_a \}_{a \in \E_\A}$ with an index set $\E_\A$ 
of an $\A$--$\A$-imprimitivity bimodule 
${}_\A\!X_\A$ is said to be 
{\it finitely self-related} 
if there exist
an 
$\E_\A\times \E_\A$-matrix 
$\tilde{A}  = [\tilde{A}(a,b)]_{a,b \in \E_\A}$
 with entries in 
 $\{0,1\}$
such that 
\begin{align}
\langle s_a \mid s_a \rangle_\A  =& \sum_{b \in \E_\A}
\tilde{A}(a,b) 
{}_\A\!\langle s_b \mid s_b\rangle, \qquad a \in \E_\A, \label{eq:3.2sa}\\
\langle \langle s_a \mid s_a \rangle_\A s_b\mid s_b \rangle_\A
  =& 
\tilde{A}(a,b) \langle s_b \mid s_b\rangle_\A, \qquad a, b \in \E_\A. \label{eq:3.2sab}
\end{align}
If an $\A$--$\A$-imprimitivity bimodule 
${}_\A\!X_\A$ has a finitely self-related right finite basis,
it is called a  
{\it finitely self-related imprimitivity bimodule.}
\end{definition}

A $C^*$-algebra itself
$\A$ 
has a natural structure of $\A$--$\A$ imprimitivity bimodule by the inner products
\begin{equation}
\!\langle x \mid y \rangle_{\A}: = x^* y, \qquad
{}_\A\!\langle x \mid y \rangle:= x y^*
\quad \text{ for }\quad x,  y \in \A. \label{eq:AAbimodule}
\end{equation}
We write the $\A$--$\A$ imprimitivity bimodule as 
${}_\A\!\A_{\A}$ and call it the identity $\A$--$\A$ imprimitivity bimodule.
If the identity $\A$--$\A$ imprimitivity  
bimodule ${}_\A\!\A_{\A}$ has a finitely self-related right finite basis
$\{s_a \}_{a \in \E_\A}$,
the $C^*$-algebra $\A$ with the basis
$\{s_a\}_{a \in \E_\A}$ 
is said to be {\it finitely self-related}.
 \begin{lemma}\label{lem:OAbasis}
Let
$A=[A(i,j)]_{i,j=1}^{N}$
be a nonnegative irreducible matrix.
Let $\{ S_a \}_{a \in E_A}$ be the generating partial isometries 
of the Cuntz--Krieger algebra $\OA$ satisfying the relations \eqref{eq:OAG}
for the matrix $A^G$.
Then $\{ S_a \}_{a \in E_A}$ is a finitely self-related right finite basis of the identity
$\OA$--$\OA$-imprimitivity bimodule ${}_\OA\!\OA_{\OA}$. 
\end{lemma}
\begin{proof}
Let
 $\langle\hspace{3mm}\mid\hspace{3mm}\rangle_A$
denote the $\OA$-valued right inner product 
$\langle\hspace{3mm}\mid\hspace{3mm}\rangle_\OA$
defined by the left one in \eqref{eq:AAbimodule}.
Since 
\begin{equation}
x = \sum_{a \in E_A} S_a S_a^* x = \sum_{a\in E_A}S_a\langle S_a \mid x\rangle_A,
\qquad x \in {}_\OA\!\OA _{\OA} =\OA,
\end{equation} 
$\{ S_a \}_{a \in E_A}$ is a right finite basis of
${}_\OA\!\OA _{\OA}$.
The second relation of  \eqref{eq:OAG} is rephrased as \eqref{eq:3.2sa}
for the matrix $\tilde{A} = A^G$.
We also have
\begin{align*}
\langle \langle S_a \mid S_a \rangle_A S_b\mid S_b \rangle_A
= & ((S_a^* S_a)S_b)^* S_b \\
= & S_b^* \sum_{a' \in E_A} A^G(a, a') S_{a'}S_{a'}^* S_b \\
= &  A^G(a, b) S_b^* S_b \\
= & A^G(a, b) \langle S_b \mid S_b\rangle_A, \qquad a, b \in E_A.
\end{align*}
Hence the basis $\{S_a\}_{a \in E_A}$
is finitely self-related. 
\end{proof}

For 
an $\A$--$\B$-imprimitivity bimodule ${}_{\A}\!X_\B$ and
a $\B$--$\A$-imprimitivity bimodule ${}_{\B}\!Y_\A$,
there is a natural method to construct 
$\A$--$\A$-imprimitivity bimodule 
${}_{\A}\!X_\B\otimes_{\B} {}_{\B}\!Y_\A$
 and
$\B$--$\B$-imprimitivity bimodule 
${}_{\B}\!Y_\A\otimes_{\A} {}_{\A}\!X_\B$
as relative tensor products of Hilbert $C^*$-bimodules as seen in
\cite[Theorem 5.9]{Rieffel1} 
(cf. \cite[Definition 1.20]{KW} and \cite[Proposition 3.16]{RW}).
The $\A$-bimodule structure of 
${}_{\A}\!X_\B\otimes_{\B} {}_{\B}\!Y_\A$
is given by
\begin{equation*}
a(x \otimes_\B y) := ax \otimes_\B y, \qquad
(x \otimes_\B y)a := x \otimes_\B ya
\quad
\text{ for }
x\in {}_{\A}\!X_\B,\,
y \in {}_{\B}\!Y_\A, \,
a \in \A.
\end{equation*}
The $\A$-valued right inner product and 
the $\A$-valued left inner product of  
${}_{\A}\!X_\B\otimes_{\B} {}_{\B}\!Y_\A$
are given by
\begin{align}
{}_\A\!\langle x_1\otimes_\B y_1 \mid x_2\otimes_\B y_2\rangle 
:=& {}_\A\!\langle x_1 {}_\B\!\langle y_1 \mid y_2\rangle \mid  x_2 \rangle,
 \label{eq:3.Atensor}\\
\langle x_1\otimes_\B y_1\mid x_2\otimes_\B y_2\rangle_\A 
:=& \langle y_1 \mid  \langle x_1 \mid x_2\rangle_\B y_2 \rangle_\A. \label{eq:3.tensorA}
\end{align}
Similarly we have 
$\B$--$\B$-imprimitivity bimodule 
${}_{\B}\!Y_\A\otimes_{\A} {}_{\A}\!X_\B$.

\begin{definition}\label{defn:EME}
Let $\A$ and $\B$ be $C^*$-algebras.
Finitely self-related 
imprimitivity bimodules 
$({}_\A\!Z^1_{\A}, \{ s_a \}_{a \in \E_\A})$
and  
$({}_\B\!Z^2_{\B}, \{ t_b \}_{b \in \E_\B})$
with
 finitely self-related  right finite bases 
$\{ s_a \}_{a \in \E_\A }$ 
and
$\{ t_b \}_{b \in \E_\B }$ 
respectively
are said to  be 
{\it basis relatedly elementary Morita equivalent}\/
if there exist 
finitely related 
$\A$--$\B$-imprimitivity bimodule 
${}_{\A}\!X_\B$
and
 $\B$--$\A$-imprimitivity bimodule
${}_{\B}\!Y_\A$ 
with finitely related right finite bases 
$\{ \eta_c\}_{c \in \E_C} \subset {}_{\A}\!X_\B$
and  $\{ \xi_d\}_{d \in \E_D} \subset {}_{\B}\!Y_\A$,
respectively
such that 
there exist
maps 
$c: \E_\A \cup \E_\B \longrightarrow \E_C, \,
 d: \E_\A \cup \E_\B \longrightarrow \E_D
$
such that
the correspondences
\begin{align}
\phi_\A: s_a\in  {}_\A\!Z^1_{\A}& \longrightarrow 
       \eta_{c(a)} \otimes_\B \xi_{d(a)} \in {}_{\A}\!X_\B\otimes_{\B}{}_{\B}\!Y_\A, 
\label{eq:phi}\\
\phi_\B: t_b \in {}_\B\!Z^2_{\B}& \longrightarrow 
      \xi_{d(b)}\otimes_\A \eta_{c(b)} \in {}_{\B}\!Y_\A\otimes_{\A}{}_{\A}\!X_\B
\label{eq:psi}
\end{align}
give rise to bimodule isomorphisms,
where a bimodule isomorphism means an isomorphism
of bimodule which preserves their left and right inner products,
respectively.
(cf. \cite[p. 57]{RW}).
\end{definition}
The isomorphisms $\phi_\A, \phi_\B$ are called 
{\it basis relatedly bimodule isomorphisms}.
If  pairs 
$({}_\A\!Z^1_{\A}, \{ s_a \}_{ a \in \E_\A })$
and  
$({}_\B\!Z^2_{\B}, \{ t_b \}_{ b \in \E_\B })$ 
are connected by a finite sequence of basis related  elementary Morita equivalences, 
then they are said to be {\it basis relatedly Morita equivalent}.

\begin{definition}
\begin{enumerate}
\renewcommand{\theenumi}{\roman{enumi}}
\renewcommand{\labelenumi}{\textup{(\theenumi)}}
\item 
A pair $(\A, \{s_a\}_{a \in E_\A})$
of a $C^*$-algebra and a finitely self-related right finite basis 
of the identity imprimitivity bimodule 
${}_\A\!\A_\A$ is called {\it a $C^*$-algebra with right finite basis}.
\item
Two $C^*$-algebras $(\A,\{s_a\}_{a \in E_\A})$ 
and $(\B,\{t_b\}_{b \in E_\B})$ with right finite bases
are said to be {\it basis relatedly elementary Morita equivalent}
if the identity imprimitivity bimodules
$({}_\A\!\A_{\A}, \{ s_a \}_{a \in \E_\A})$
and  
$({}_\B\!\B_{\B}, \{ t_b \}_{b \in \E_\B})$
are basis relatedly elementary Morita equivalent.
They are said to 
 {\it basis relatedly Morita equivalent}\/ 
 if $({}_\A\!\A_{\A}, \{ s_a \}_{ a \in \E_\A})$
and  
$({}_\B\!\B_{\B}, \{ t_b \}_{ b \in \E_\B })$
are basis relatedly Morita equivalent.
\end{enumerate}
\end{definition}
By Lemma \ref{lem:OAbasis}, 
the Cuntz--Krieger algebra $(\OA,\{S_a\}_{a \in E_A})$ 
 with the generating partial isometries $\{S_a\}_{a \in E_A}$
satisfying \eqref{eq:OAG} 
is a $C^*$-algebra with right finite basis. 
The right finite basis 
$\{S_a\}_{a \in E_A}$
is called the {\it canonical}\/ right finite basis of $\OA$.
\begin{theorem}\label{thm:main1}
Assume that nonnegative square irreducible matrices $A$ and $B$ are elementary equivalent such that
$A  = CD, B = DC$
for some nonnegative rectangular matrices $C,D$.
Then the Cuntz--Krieger algebras  
$(\OA, \{S_a\}_{a\in E_A})$ and 
$(\OB, \{S_b\}_{b\in E_B})$ 
with the canonical  right finite bases
are basis relatedly elementary Morita equivalent.
\end{theorem}
\begin{proof}
Suppose that two nonnegative square matrices $A$ and $B$ 
are elementary equivalent
such that $A =CD$ and $B = DC$.
The sizes of the matrices
$A$ and $B$ are denoted by
$N$ and $M$ respectively,
so that 
$C$ is an $N \times M$ matrix and 
$D$ is an $M \times N$ matrix, respectively. 
 We set the square matrix
$
Z =
\begin{bmatrix}
0 & C \\
D & 0
\end{bmatrix}
$
as a block matrix, and we see
\begin{equation*}
Z^2 =
\begin{bmatrix}
CD & 0 \\
0 & DC
\end{bmatrix}
=
\begin{bmatrix}
A & 0 \\
0 & B
\end{bmatrix}.
\end{equation*}
Similarly to \label{eq:AG}
the directed graph 
$G_A =(V_A,E_A)$
for the matrix $A$,
let  us denote by
$
G_B =(V_B,E_B),
G_C =(V_C, E_C),
G_D =(V_D, E_D)
$
and
$
G_Z =(V_Z,E_Z)
$
the associated directed graphs to the nonnegative matrices
$ B, C, D$ and $Z$, respectively,
so that
$V_C = V_D = V_Z = V_A \cup V_B$ and $E_Z = E_C \cup E_D$.  
By the equalities $A=CD$ and $B = DC$,
we may take and fix bijections
$\varphi_{A,CD}$ from $E_A$ to a subset of $E_C \times E_D$
and
$\varphi_{B,DC}$ from $E_B$ to a subset of $E_D \times E_C$.
Let 
$ S_c, S_d, c \in E_C, d \in E_D$ 
be the generating partial isometries of 
the Cuntz--Krieger algebra 
$\OZ$
for the matrix $Z$, so that
$
\sum_{c \in E_C} S_c S_c^*
+
\sum_{d \in E_D} S_d S_d^*
=1$
and
\begin{align}
S_c^* S_c & = \sum_{d \in E_D}Z^G(c,d) S_d S_d^*
= \sum_{d \in E_D}C^G(c,d) S_d S_d^*, 
\qquad c \in E_C, \label{eq:3.7Sc}\\
S_d^* S_d & = \sum_{c \in E_C}Z^G(d,c) S_c S_c^*
= \sum_{c \in E_C}D^G(d,c) S_c S_c^*
\qquad d \in E_D, \label{eq:3.7Sd}
\end{align}
for $c \in E_C, d \in E_D$,
where $Z^G, C^G, D^G$ are the matrices with entries in $\{0,1\}$
defined by  similar ways to (1.1) 
from the directed graphs 
$G_Z, G_C, G_D$, respectively. 
Since 
$S_c S_d \ne 0$ 
(resp. $S_d S_c \ne 0$) 
if and only if 
$\varphi_{A,CD}(a) = cd$
(resp.
$\varphi_{B,DC}(b) = dc$)
 for some $a \in E_A$
(resp. $b \in E_B$),
we may identify $cd$ (resp. $dc$) 
with $a$ (resp. $b$)
through the map 
$\varphi_{A,CD}$ (resp. $\varphi_{B,DC}$).
In this case,
we may write $a = c(a) d(a) \in E_C\times E_D$
(resp. $b = d(b) c(b) \in E_D\times E_C$)
if $\varphi_{A,CD}(a) = cd$
(resp. $\varphi_{B,DC}(b) = dc$)
to have maps
\begin{equation*}
c : E_A\cup E_B \longrightarrow E_C,
\qquad
d : E_A\cup E_B \longrightarrow E_D.
\end{equation*}
We may then write 
$S_{c(a) d(a)} = S_a$
(resp. $S_{d(b) c(b)} = S_b$)
where $S_{c(a) d(a)}$ denotes $S_{c(a)} S_{d(a)}$ 
(resp. $S_{d(b) c(b)}$ denotes $S_{d(b)} S_{c(b)}$).
Let us  define two particular projections $P_A$ and $P_B$ in $\OZ$ by 
$P_A = \sum_{c \in E_C} S_c S_c^*$
and
$P_B = \sum_{d \in E_D} S_d S_d^*$
so that $P_A + P_B =1$.
It has been shown in \cite{MaETDS2004} (cf. \cite{MaMZ2016})
that 
\begin{equation}
P_A \OZ P_A = \OA, \qquad P_B \OZ P_B = \OB.
\label{eq:4.1}
\end{equation}
We set 
\begin{equation}
{}_{A}\!X_B = P_A \OZ P_B, \qquad
{}_{B}\!Y_A = P_B \OZ P_A.
\end{equation}
We equip $\OZ$ with the structure of  the identity $\OZ$--$\OZ$-imprimitivity bimodule 
defined by 
\eqref{eq:AAbimodule}.
As the submodules of ${}_\OZ\!\OZ_{\OZ}$,
${}_{A}\!X_B$ and
${}_{B}\!Y_A$ 
 become 
$P_A \OZ P_A$--$P_B\OZ P_B$-imprimitivity bimodule
and
$P_B \OZ P_B$--$P_A\OZ P_A$-imprimitivity bimodule,
respectively.
As in \cite[Lemma 3.1]{MaETDS2004},
both projections $P_A$ and $P_B$ are full projections so that  
${}_{A}\!X_B$
and
${}_{B}\!Y_A$
are  
$\OA$--$\OB$ and $\OB$--$\OA$ imprimitivity bimodules.
We set
\begin{equation}
\eta_c = S_c \quad \text{ for } c \in E_C,
\qquad
\xi_d = S_d \quad \text{ for } d \in E_D.
\end{equation}
The relations \eqref{eq:3.7Sc} and \eqref{eq:3.7Sd}
 imply the equalities
\eqref{eq:3.2eta} and \eqref{eq:3.2xi}, respectively
for $\tilde{C} = C^G$ and $\tilde{D} = D^G$.
We denote by
$\langle\hspace{3mm}\mid\hspace{3mm}\rangle_A$
and
$\langle\hspace{3mm}\mid\hspace{3mm}\rangle_B$
the $\OA$-valued right inner product 
$\langle\hspace{3mm}\mid\hspace{3mm}\rangle_\OA$
in ${}_{B}\!Y_A$ and
the $\OB$-valued right inner product 
$\langle\hspace{3mm}\mid\hspace{3mm}\rangle_\OB$
in ${}_{A}\!X_B$,
respectively.
We also have
\begin{align*}
\langle \langle \eta_c \mid \eta_c \rangle_B \xi_d
\mid \xi_d \rangle_A 
& =( (S_c^* S_c)S_d )^* S_d \\
& = S_d^* S_c^* S_c S_d \\
& = S_d^*(\sum_{d' \in E_D}C^G(c,d') S_{d'} S_{d'}^*) S_d \\
& = C^G(c,d) S_{d}^* S_{d} \\
& = C^G(c,d) \langle \xi_d \mid \xi_d \rangle_A. 
\end{align*}
Hence we have the equality 
\eqref{eq:3.2etaxi}
for $\tilde{C} = C^G$,
and similarly 
\eqref{eq:3.2xieta}
for $\tilde{D} = D^G$.
Therefore 
${}_{A}\!X_B$
and
${}_{B}\!Y_A$
are finitely related imprimitivity bimodules.
We note that the $\OB$--$\OA$-imprimitivity bimodule
${}_{B}\!Y_A =P_B\OZ P_A$  
is the conjugate bimodule
of ${}_{A}\!X_B =P_A\OZ P_B$, because
 the correspondence
$$
\varPhi: \bar{x} \in \overline{P_A\OZ P_B}
\longrightarrow
x^* \in P_B\OZ P_A
$$
yields an isomorphism of imprimitivity bimodules.
By \cite[Proposition 3.28]{RW},
 the correspondence
\begin{equation}
\Psi: x\otimes_\OB y^* \in {}_{A}\!X_B\otimes_{\OB}{}_{B}\!Y_A
\longrightarrow
{}_A\!\langle x,y \rangle \in
\OA \quad \text{ for }
x,y \in {}_{A}\!X_B = P_A\OZ P_B
\end{equation}
yields an isomorphism
between the $\OA$--$\OA$-imprimitive bimodules
${}_{A}\!X_B\otimes_{\OB}{}_{B}\!Y_A$
and
${}_{\OA}\!{\OA}_{\OA}=\OA.$
As
$$
\Psi(\eta_{c(a)} \otimes_{\OB}\xi_{d(a)})
=
\Psi(S_{c(a)} \otimes_{\OB} S_{d(a)})
={}_A\!\langle S_{c(a)}, S_{d(a)}^* \rangle
= S_{c(a)}S_{d(a)}
= S_a
$$
for $a \in E_A$,
the correspondence
\begin{equation*}
S_a \in \OA  \longrightarrow \eta_{c(a)} \otimes_\OB \xi_{d(a)} \in 
{}_{A}\!X_B\otimes_{\OB}{}_{B}\!Y_A
\end{equation*}
gives rise to a basis related bimodule isomorphism between
${}_{A}\!X_B\otimes_{\OB}{}_{B}\!Y_A$
and
${}_{\OA}\!{\OA}_{\OA}=\OA.$
We similarly know that 
\begin{equation*}
S_b \in \OB  \longrightarrow \xi_{d(b)} \otimes_\OA \eta_{c(b)} 
\in {}_{B}\!Y_A\otimes_{\OA}{}_{A}\!X_B    
\end{equation*}
gives rise to a basis related  bimodule isomorphism
between 
${}_{B}\!Y_A\otimes_{\OA}{}_{A}\!X_B$
and
${}_{\OB}\!{\OB}_{\OB}=\OB.$
Therefore  
the Cuntz--Krieger algebras  
$(\OA, \{S_a\}_{a\in E_A})$ and 
$(\OB, \{S_b\}_{b\in E_B})$ 
with the canonical right finite bases
are basis relatedly elementary Morita equivalent.
\end{proof}
As seen in the above proof,
under the condition that $A=CD, B =DC$,
one may take ${}_B\!Y_A$ as the conjugate bimodule of
${}_A\!X_B$, 
so that one obtains the following corollary.
\begin{corollary}
Assume that nonnegative square irreducible matrices 
$A, B$ are elementary equivalent such that
$A=CD, B=DC$.
Then there exists a bipartitely related $\OA$--$\OB$-imprimitivity
bimodule 
$({}_A\!X_B, \{S_c\}_{c \in E_C}, \{S^*_d\}_{d \in E_D}, C^G, D^G)$
such that 
$({}_A\!X_B, \{S_c\}_{c \in E_C})$
and its conjugate
$(\overline{{}_A\!X_B}, \{S_d\}_{d \in E_D})$   
give rise to a basis related elementary Morita equivalence
between
$(\OA, \{S_a\}_{a \in E_A})$ and
$(\OB, \{S_b\}_{b \in E_B})$.
\end{corollary}

\section{Converse implication  of Theorem \ref{thm:main1}}
In this section, we will show the converse implication of Theorem \ref{thm:main1}.
Throughout this section, we assume that 
the Cuntz--Krieger algebras
$(\OA, \{S_a \}_{ a \in E_A})$  
and 
$(\OB,\{S_b \}_{b \in E_B })$
with their canonical right finite bases
are basis relatedly elementary Morita equivalent 
by 
finitely related $\OA$--$\OB$-imprimitivity bimodule
${}_\OA\!X_\OB$ written ${}_A\! X_B$ with right finite basis $\{\eta_c\}_{c \in \E_C}$
and 
$\OB$--$\OA$-imprimitivity bimodule
${}_\OB\!Y_\OA$ written ${}_B\!Y_A$ with right finite basis $\{\xi_d\}_{d \in \E_D}$,
where
$\E_C$ and $\E_D$ are finite index sets, respectively.
Let us denote by 
$\langle\hspace{3mm}\mid\hspace{3mm}\rangle_A$
and
${}_A\!\langle\hspace{3mm}\mid\hspace{3mm}\rangle$
the $\OA$-valued right inner product of  ${}_B\!Y_A$
and
the $\OA$-valued left inner product of  ${}_A\!X_B$,
respectively.
Similar notations 
$\langle\hspace{3mm}\mid\hspace{3mm}\rangle_B$
and
${}_B\!\langle\hspace{3mm}\mid\hspace{3mm}\rangle$
will be used.
By Definition \ref{defn:EME}, one may take  
maps 
$c: E_A \cup E_B \longrightarrow \E_C, \,
 d: E_A \cup E_B \longrightarrow \E_D
$
such that
the correspondences
\begin{align}
\phi_A: S_a\in  {}_\OA\!\OA_{\OA}& \longrightarrow 
       \eta_{c(a)} \otimes_\OB \xi_{d(a)} \in {}_{A}\!X_B\otimes_{\OB}{}_{B}\!Y_A, 
\label{eq:phiA}\\
\phi_B: t_b \in {}_\OB\!\OB_{\OB}& \longrightarrow 
      \xi_{d(b)}\otimes_\OB \eta_{c(b)} \in {}_{B}\!Y_A\otimes_{\OA}{}_{A}\!X_B
\label{eq:phiB}
\end{align}
give rise to bimodule isomorphisms.
We take $\E_C \times \E_D$ matrix $\tilde{C} =[\tilde{C}(c,d)]_{c \in \E_C, d \in \E_D}$
and     $\E_D \times \E_C$ matrix $\tilde{D} =[\tilde{D}(d,c)]_{d \in \E_D, c \in \E_C}$
satisfying
\eqref{eq:3.2eta}, \eqref{eq:3.2xi}, \eqref{eq:3.2etaxi}, \eqref{eq:3.2xieta}.
The relative tensor product ${}_{A}\!X_B\otimes_{\OB}{}_{B}\!Y_A$
is written ${}_{A}\!X_B\otimes_{B}{}_{B}\!Y_A$,
and similar notation 
${}_{B}\!Y_A\otimes_{A}{}_{A}\!X_B$ will be used.
We first note the following lemma.
\begin{lemma}\label{lem:4.1}
Keep the above assumptions and notations.
\begin{enumerate}
\renewcommand{\theenumi}{\roman{enumi}}
\renewcommand{\labelenumi}{\textup{(\theenumi)}}
\item 
$\tilde{C}(c(a),d(a)) =1$ for all $a \in E_A$, 
and hence
\begin{align}
\langle S_a \mid S_a \rangle_A 
&=\langle \eta_{c(a)} \otimes_B \xi_{d(a)} \mid \eta_{c(a)} \otimes_B \xi_{d(a)} \rangle_A
 =\langle \xi_{d(a)} \mid \xi_{d(a)} \rangle_A, \label{eq:SaA}\\ 
{}_A\!\langle S_a \mid S_a \rangle 
&={}_A\!\langle \eta_{c(a)} \otimes_B \xi_{d(a)} \mid \eta_{c(a)} \otimes_B \xi_{d(a)} \rangle
  \quad \text{ for } a \in E_A. \label{eq:ASa}
\end{align}
\item 
$\tilde{D}(d(b),c(b)) =1$ for all $b \in E_B$,
and hence
\begin{align}
\langle S_b \mid S_b \rangle_B 
&=\langle \xi_{d(b)}\otimes_A  \eta_{c(b)} \mid \xi_{d(b)}\otimes_A  \eta_{c(b)} \rangle_B
=\langle \eta_{c(b)} \mid \eta_{c(b)} \rangle_B, \label{eq:SbB}\\
{}_B\!\langle S_b \mid S_b \rangle 
&={}_B\!\langle 
\xi_{d(b)}\otimes_A  \eta_{c(b)} \mid \xi_{d(b)}\otimes_A  \eta_{c(b)} \rangle
\quad \text{ for } b \in E_B. \label{eq:BSb}
 \end{align}
\end{enumerate}
\end{lemma}
\begin{proof}
(i)
Since the map
$\phi_A: S_a \in  {}_\OA\!\OA_{\OA} \longrightarrow 
       \eta_{c(a)} \otimes_B \xi_{d(a)} \in {}_{A}\!X_B\otimes_{B}{}_{B}\!Y_A
$
in \eqref{eq:phiA}
gives rise to a bimodule isomorphism,
we have
\begin{equation*}
\langle S_a \mid S_a \rangle_A 
= \langle \phi_A(S_a) \mid \phi_A(S_a) \rangle_A 
= \langle \eta_{c(a)} \otimes_B \xi_{d(a)} 
   \mid \eta_{c(a)} \otimes_B \xi_{d(a)} \rangle_A.
\end{equation*}
By the definition of inner products of relative tensor products
in \eqref{eq:3.Atensor}, \eqref{eq:3.tensorA}
and the equality \eqref{eq:3.2etaxi},
we have
\begin{align*}
\langle \eta_{c(a)} \otimes_B \xi_{d(a)} \mid \eta_{c(a)} \otimes_B \xi_{d(a)} \rangle_A
=&\langle \xi_{d(a)} \mid 
 \langle \eta_{c(a)} \mid \eta_{c(a)}\rangle_B \xi_{d(a)} \rangle_A \\
=&\langle  \langle \eta_{c(a)} \mid \eta_{c(a)}\rangle_B \xi_{d(a)} \mid 
 \xi_{d(a)} \rangle_A \\
=&\tilde{C}(c(a), d(a)) \langle \xi_{d(a)} \mid  \xi_{d(a)} \rangle_A
\end{align*} 
so that 
\begin{equation*}
\langle S_a \mid S_a \rangle_A 
= \tilde{C}(c(a), d(a)) \langle \xi_{d(a)} \mid  \xi_{d(a)} \rangle_A.
\end{equation*}
Since
$\langle S_a \mid S_a \rangle_A = S_a^* S_a \ne 0$,
we have
$\tilde{C}(c(a), d(a)) \langle \xi_{d(a)} \mid  \xi_{d(a)} \rangle_A \ne 0$.
As $\tilde{C}(c(a), d(a))=0$ or $1$,
we conclude
$\tilde{C}(c(a), d(a))=1$
and hence 
\begin{equation*}
\langle S_a \mid S_a \rangle_A 
= \langle \xi_{d(a)} \mid  \xi_{d(a)} \rangle_A.
\end{equation*}
We thus have \eqref{eq:SaA}.
As
the map
$\phi_A: S_a \in  {}_\OA\!\OA_{\OA} \longrightarrow 
       \eta_{c(a)} \otimes_B \xi_{d(a)} \in {}_{A}\!X_B\otimes_{B}{}_{B}\!Y_A
$
in \eqref{eq:phiA}
gives rise to a bimodule isomorphism,
the equality
\eqref{eq:ASa}
holds.

(ii) is similarly shown to (i).
\end{proof}

\begin{lemma}\label{lem:4.2}
\hspace{6cm}
\begin{enumerate}
\renewcommand{\theenumi}{\roman{enumi}}
\renewcommand{\labelenumi}{\textup{(\theenumi)}}
\item 
$\tilde{C}(c(a),d) =0$ implies $\langle \xi_d \mid \xi_{d(a)}\rangle_A =0.$
Hence we have
\begin{equation}
\sum_{d \in \E_D}\tilde{C}(c(a),d) {}_B\!\langle \xi_d \mid \xi_d\rangle \xi_{d(a)} =\xi_{d(a)}.
\end{equation}
\item 
$\tilde{D}(d(b),c) =0$ implies $\langle \eta_c \mid \eta_{c(b)}\rangle_B =0.$
Hence we have
\begin{equation}
\sum_{c \in \E_C}\tilde{D}(d(b),c) {}_A\!\langle \eta_c \mid \eta_c\rangle \eta_{c(b)} =\eta_{c(b)}.
\end{equation}
\end{enumerate}
\end{lemma}
\begin{proof}
(i)
By \eqref{eq:3.2eta},
we have the following equalities:
\begin{align*}
\langle S_a \mid S_a\rangle_A 
=& \langle \xi_{d(a)} \mid \langle \eta_{c(a)} \mid 
\eta_{c(a)} \rangle_B \xi_{d(a)}\rangle_A 
\\
=& \sum_{d \in \E_D}\tilde{C}(c(a),d) \langle \xi_{d(a)} \mid 
{}_B\!\langle \xi_d \mid \xi_d\rangle  
\xi_{d(a)}\rangle_A 
\\
=& \sum_{d \in \E_D}\tilde{C}(c(a),d) \langle \xi_{d(a)} \mid 
\xi_d \langle \xi_d \mid   
\xi_{d(a)}\rangle_A \rangle_A
\\
=& \sum_{d \in \E_D}\tilde{C}(c(a),d) 
(\langle \xi_d \mid \xi_{d(a)}\rangle_A)^* 
\langle \xi_d \mid  \xi_{d(a)}\rangle_A, \\
\intertext{and}
\langle \xi_{d(a)} \mid \xi_{d(a)}\rangle_A
= & \langle \xi_{d(a)} \mid 
\sum_{d \in \E_D}\xi_d \langle \xi_d \mid \xi_{d(a)}\rangle_A\rangle_A \\
= & \sum_{d \in \E_D}
(\langle \xi_d \mid \xi_{d(a)}\rangle_A)^*
\langle \xi_d \mid \xi_{d(a)}\rangle_A.
 \end{align*}
By Lemma \ref{lem:4.1},
we have
$\langle S_a \mid S_a\rangle_A = \langle \xi_{d(a)} \mid \xi_{d(a)}\rangle_A$
so that 
\begin{equation}
\sum_{d \in \E_D}\tilde{C}(c(a),d) 
(\langle \xi_d \mid \xi_{d(a)}\rangle_A)^* 
\langle \xi_d \mid  \xi_{d(a)}\rangle_A
=
\sum_{d \in \E_D}
(\langle \xi_d \mid \xi_{d(a)}\rangle_A)^*
\langle \xi_d \mid \xi_{d(a)}\rangle_A.
\end{equation}
Since
$\tilde{C}(c(a),d) =0$ or $1$,
we get
\begin{equation}
\tilde{C}(c(a),d) 
\langle \xi_d \mid  \xi_{d(a)}\rangle_A
=
\langle \xi_d \mid \xi_{d(a)}\rangle_A
\qquad\text{ for all }
d \in \E_D. \label{eq:Ccad}
\end{equation}
This implies that 
$\langle \xi_d \mid \xi_{d(a)}\rangle_A=0$
if $\tilde{C}(c(a),d)=0$.

By \eqref{eq:Ccad}, we have
\begin{align*}
\sum_{d \in \E_D}\tilde{C}(c(a),d) 
{}_B\!\langle \xi_d \mid \xi_d\rangle  \xi_{d(a)}
=&
\sum_{d \in \E_D}\tilde{C}(c(a),d) 
\xi_d \langle \xi_d \mid  \xi_{d(a)}\rangle_A \\
=&
\sum_{d \in \E_D} 
\xi_d \langle \xi_d \mid \xi_{d(a)}\rangle_A
= \xi_{d(a)}.
\end{align*}

(ii) is similarly shown to (i).
\end{proof}
The above lemma immediately implies the following lemma.
\begin{lemma}\label{lem:4.4}
\hspace{6cm}
\begin{enumerate}
\renewcommand{\theenumi}{\roman{enumi}}
\renewcommand{\labelenumi}{\textup{(\theenumi)}}
\item 
If $(c,d) \ne (c(a),d(a)) $ for any $a \in E_A$,
we have $\eta_c\otimes_B \xi_d =0$.
\item 
If 
$(d,c) \ne (d(b),c(b)) $ for any $b \in E_B$,
we have $\xi_d \otimes_A \eta_c =0$.
\end{enumerate}
\end{lemma}
\begin{proof}
(i)
By Lemma \ref{lem:4.1},
\begin{equation*}
{}_A\!\langle S_a \mid S_a \rangle = 
{}_A\!\langle \eta_{c(a)} \otimes_B \xi_{d(a)} \mid \eta_{c(a)} \otimes_B \xi_{d(a)}\rangle,
\qquad a \in E_A
\end{equation*}
so that 
\begin{equation*}
1 = \sum_{a \in E_A} S_a S_a^* 
  = \sum_{a \in E_A} {}_A\!\langle \eta_{c(a)} \otimes_B \xi_{d(a)} 
\mid \eta_{c(a)} \otimes_B \xi_{d(a)}\rangle.
\end{equation*}
On the other hand, by the equality
$ \sum_{d \in \E_D} {}_\B\!\langle \xi_d \mid \xi_d \rangle = 1$ in $\OB$,
we have
\begin{align*}
 \sum_{c \in \E_C}\sum_{d \in \E_D}
{}_A\!\langle \eta_{c} \otimes_B \xi_{d} \mid \eta_{c} \otimes_B \xi_{d}\rangle 
= &\sum_{c \in \E_C}\sum_{d \in \E_D}
{}_A\!\langle \eta_{c}  {}_B\langle \xi_{d} \mid \xi_d\rangle \mid
 \eta_{c} \rangle \\
=  &\sum_{c \in \E_C}
{}_A\!\langle \eta_{c} \sum_{d \in \E_D}{}_B\langle \xi_{d} \mid \xi_d\rangle \mid
 \eta_{c} \rangle \\
= &  \sum_{c \in \E_C}
{}_A\!\langle \eta_{c}  \mid
 \eta_{c} \rangle 
= 1
\end{align*}
so that 
\begin{equation*}
1  = \sum_{a \in E_A}
 {}_A\!\langle \eta_{c(a)} \otimes_B \xi_{d(a)} 
\mid \eta_{c(a)} \otimes_B \xi_{d(a)}\rangle
= \sum_{c \in \E_C}\sum_{d \in \E_D}
{}_A\!\langle \eta_{c} \otimes_B \xi_{d} 
\mid \eta_{c} \otimes_B \xi_{d}\rangle. 
\end{equation*}
As ${}_A\!\langle \eta_{c} \otimes_B \xi_{d} 
\mid \eta_{c} \otimes_B \xi_{d}\rangle \ge 0$,
we have
${}_A\!\langle \eta_{c} \otimes_B \xi_{d} \mid \eta_{c} \otimes_B \xi_{d}\rangle =0$
for $c \in \E_C, d \in \E_D$
with
 $(c,d)  \ne (c(a),d(a)) $ for any $a \in E_A$.
This means that
$\eta_c\otimes_B \xi_d =0$
if $(c,d)  \ne (c(a),d(a)) $ for any $a \in E_A$.

(ii) is similarly shown to (i).
\end{proof}

\begin{lemma}\label{lem:4.5}
\hspace{6cm}
\begin{enumerate}
\renewcommand{\theenumi}{\roman{enumi}}
\renewcommand{\labelenumi}{\textup{(\theenumi)}}
\item 
$A^G(a_i,a_j) = \tilde{D}(d(a_i), c(a_j))$ for all $a_i, a_j \in E_A$.
\item 
$B^G(b_k, b_l) = \tilde{C}(c(b_k), d(b_l))$ for all $b_k, b_l \in E_B$.
\end{enumerate}
\end{lemma}
\begin{proof}
(i)
The second operator relations of \eqref{eq:OAG}
is rewritten in ${}_\OA\!\OA_{\OA}$ such as:
\begin{equation}
\langle S_{a_i} \mid S_{a_i} \rangle_A
 = \sum_{j=1}^{N_A} A^G(a_i,a_j) 
{}_A\!\langle S_{a_j} \mid S_{a_j}\rangle, \qquad i=1,\dots,N_A. \label{eq:rewritten}
\end{equation}
By \eqref{eq:3.2xi}, we have
\begin{align*}
\langle S_{a_i} \mid S_{a_i} \rangle_A
=& \langle  \xi_{d(a_i)} 
   \mid 
   \xi_{d(a_i)}\rangle_A \\
= &\sum_{c \in \E_C}\tilde{D}(d(a_i),c)
  {}_A\!\langle \eta_{c} \mid \eta_{c}\rangle \\
= &\sum_{c \in \E_C}\tilde{D}(d(a_i),c)\sum_{d \in \E_D}
  {}_A\!\langle \eta_{c} {}_B\langle\xi_d\mid\xi_d\rangle
 \mid \eta_{c}\rangle \\
= &\sum_{c \in \E_C}\sum_{d \in \E_D}\tilde{D}(d(a_i),c)
{}_A\!\langle \eta_{c} \otimes_B \xi_{d} \mid \eta_{c} \otimes_B \xi_{d}\rangle. 
\end{align*}
By Lemma \ref{lem:4.4}, we know
$
{}_A\!\langle \eta_{c} \otimes_B \xi_{d} \mid \eta_{c} \otimes_B \xi_{d}\rangle
=0
$
if 
$(c,d) \ne (c(a_j), d(a_j))$
for any $a_j \in E_A$,
so that we have by Lemma \ref{lem:4.1}
\begin{align}
\langle S_{a_i} \mid S_{a_i} \rangle_A
= & \sum_{a_j  \in E_A}\tilde{D}(d(a_i),c(a_j))
{}_A\!\langle \eta_{c(a_j)} \otimes_B \xi_{d(a_j)} 
\mid \eta_{c(a_j)} \otimes_B \xi_{d(a_j)}\rangle \\
= & \sum_{a_j  \in E_A}\tilde{D}(d(a_i),c(a_j))
{}_A\!\langle S_{a_j} \mid S_{a_j}\rangle.
\label{eq;sai} 
\end{align}
By \eqref{eq:rewritten}, \eqref{eq;sai},
we have
\begin{equation}
\sum_{a_j  \in E_A}\tilde{D}(d(a_i),c(a_j))
S_{a_j}S_{a_j}^*
=
\sum_{a_j  \in E_A} A^G(a_i,a_j) 
S_{a_j}S_{a_j}^*
\end{equation}
so that 
$
\tilde{D}(d(a_i),c(a_j)) = A^G(a_i,a_j)
$
for all $a_i, a_j \in E_A$.
\end{proof}
\begin{lemma}\label{lem:4.6}
\hspace{5cm} 
\begin{enumerate}
\renewcommand{\theenumi}{\roman{enumi}}
\renewcommand{\labelenumi}{\textup{(\theenumi)}}
\item 
For $a_i, a_j \in E_A$, we have $\tilde{D}(d(a_i), c(a_j)) =1$ 
if and only if 
there exists $b_k \in E_B$ such that $d(a_i) = d(b_k)$ and $c(a_j) = c(b_k)$.
In this case such $b_k$ is unique.
\item 
For  $b_k, b_l \in E_B$, we have
$\tilde{C}(c(b_k), d(b_l))=1$
if and only if  
there exists $a_j \in E_A$ such that $c(b_k) = c(a_j)$ and  $d(b_l) = d(a_j)$.
In this case such $a_j$ is unique.
\end{enumerate}
\end{lemma}
\begin{proof}
(i) 
Suppose that there exists $b_k \in E_B$ such that $d(a_i) = d(b_k)$ and $c(a_j) = c(b_k)$.
By Lemma \ref{lem:4.1} (ii), we have
$\tilde{D}(d(a_i), c(a_j)) =1$.

Conversely,  assume that $\tilde{D}(d(a_i), c(a_j)) =1$.
By  \eqref{eq:3.2xieta} and \eqref{eq:3.tensorA},
we have
\begin{align*}
\langle \xi_{d(a_i)}\otimes_A  \eta_{c(a_j)} 
\mid \xi_{d(a_i)}\otimes_A  \eta_{c(a_j)} \rangle_B 
= & \tilde{D}(d(a_i),c(a_j)) \langle \eta_{c(a_j)} \mid \eta_{c(a_j)} \rangle_B \\
= &  \langle \eta_{c(a_j)} \mid \eta_{c(a_j)} \rangle_B 
\ne 0.
\end{align*}
Hence we have
$\xi_{d(a_i)}\otimes_A  \eta_{c(a_j)}\ne 0$.
By Lemma \ref{lem:4.4} (ii)
there exists 
$b_k \in E_B$ such that $d(a_i) = d(b_k)$ and $c(a_j) = c(b_k)$.
Since 
$\phi_B: S_b \in {}_\OB\!\OB_\OB \longrightarrow \xi_{d(b)}\otimes_A \eta_{c(b)}
\in {}_B\!Y_A\otimes_A {}_A\!X_B$
yields a bijective isomorphism, such $b_k$ is unique.

(ii) is similar to (i).
\end{proof}
It is well-known that the relative tensor products of  Hilbert $C^*$-bimodules 
are associative in the sense that 
$({}_A\!X_B\otimes_B{}_B\!Y_A) \otimes_A{}_A\!X_B$
is naturally isomorphic to
${}_A\!X_B\otimes_B({}_B\!Y_A \otimes_A{}_A\!X_B)$
so that one may identify the vectors
$(\eta_c\otimes_B \xi_d)\otimes_A\eta_{c'}$
with
$\eta_c\otimes_B(\xi_d\otimes_A\eta_{c'})$,
that will be written 
$\eta_c\otimes_B \xi_d\otimes_A\eta_{c'}.$
Similar notation
$\xi_d\otimes_A\eta_c\otimes_B \xi_{d'}$
will be used. 
\begin{lemma}
\label{lem:4.7}
Let
 $\eta_c, \eta_{c'} \in {}_A\!X_B$ for $c, c' \in \E_C$, and
 $\xi_d, \xi_{d'} \in {}_B\!Y_A$ for $d, d' \in \E_D$. 
\begin{enumerate}
\renewcommand{\theenumi}{\roman{enumi}}
\renewcommand{\labelenumi}{\textup{(\theenumi)}}
\item 
$\eta_c\otimes_B \xi_d\otimes_A\eta_{c'} \ne 0$ 
if and only if $\tilde{C}(c,d) = \tilde{D}(d,c') = 1$.
\item 
$\xi_d\otimes_A\eta_c\otimes_B \xi_{d'} \ne 0$ 
if and only if $\tilde{D}(d,c)=\tilde{C}(c,d')  = 1$.
\end{enumerate}
\end{lemma}
\begin{proof}
(i) 
We have the following equalities:
\begin{align*}
\langle (\eta_c\otimes_B \xi_d) \otimes_A\eta_{c'}
\mid    (\eta_c\otimes_B \xi_d) \otimes_A\eta_{c'}
\rangle_B
=&  \langle \eta_{c'} \mid 
     \langle \langle \eta_c \mid \eta_c \rangle_B \xi_d \mid
               \xi_d \rangle_A 
     \eta_{c'} \rangle_B \\
=&  \tilde{C}(c,d) 
      \langle \eta_{c'} \mid 
     \langle  \xi_d \mid
               \xi_d \rangle_A 
     \eta_{c'} \rangle_B \\
=& \tilde{C}(c,d) \tilde{D}(d,c') \langle \eta_{c'} \mid \eta_{c'}
      \rangle_B.
\end{align*}
Hence we have 
$\eta_c\otimes_B \xi_d\otimes_A\eta_{c'} \ne 0$ 
if and only if $\tilde{C}(c,d) = \tilde{D}(d,c') = 1$.

(ii) By a similar way to the above computation, we have 
\begin{equation*}
\langle (\xi_d\otimes_A\eta_c)\otimes_B \xi_{d'} 
   \mid (\xi_d\otimes_A\eta_c)\otimes_B \xi_{d'} \rangle_A
=
\tilde{D}(d,c) \tilde{C}(c,d') \langle \xi_{d'} \mid \xi_{d'}
      \rangle_A.
\end{equation*}
Hence we have 
$\xi_d\otimes_A\eta_c\otimes_B \xi_{d'} \ne 0$ 
if and only if $\tilde{D}(d,c)=\tilde{C}(c,d')  = 1$.
\end{proof}
By using the above lemma, we provide the following two lemmas.
\begin{lemma}
\label{lem:4.8}
\hspace{6cm}
\begin{enumerate}
\renewcommand{\theenumi}{\roman{enumi}}
\renewcommand{\labelenumi}{\textup{(\theenumi)}}
\item 
For $c \in \E_C$, suppose that 
\begin{align}
\eta_c\otimes_B (\xi_{d(b_k)}\otimes_A\eta_{c(b_k)}) \ne 0
& \text{ for some } b_k \in E_B,  \label{eq:egek}\\
\eta_c\otimes_B (\xi_{d(b_l)}\otimes_A\eta_{c(b_l)}) \ne 0
& \text{ for some } b_l \in E_B. \label{eq:egel}
\end{align}
Then we have $s(b_k) = s(b_l)$ in $V_B$.
\item 
For $d \in \E_D$, suppose that 
\begin{align}
\xi_d\otimes_A( \eta_{c(a_i)}\otimes_B \xi_{d(a_i)})  \ne 0
& \text{ for some } a_i \in E_A, \label{eq:gegi}\\
\xi_d\otimes_A( \eta_{c(a_j)}\otimes_B \xi_{d(a_j)})  \ne 0
& \text{ for some } a_j \in E_A. \label{eq:gegj}
\end{align}
Then we have $s(a_i) = s(a_j)$ in $V_A$.
\end{enumerate}
\end{lemma}
\begin{proof}
(i)
Assume that the equalities \eqref{eq:egek} and \eqref{eq:egel} hold.
By Lemma \ref{lem:4.7},
we have 
$\tilde{C}(c,d(b_k)) = \tilde{C}(c,d(b_l)) =1.$
Since
$\varphi_B: S_b \in \OB \longrightarrow  \xi_{d(b)} \otimes \eta_{c(b)}
\in {}_B\!Y_A\otimes_A {}_A\!X_B$
yields a bimodule isomorphism, we know that 
$c : E_B \longrightarrow \E_C$ is surjective, 
so that 
there exists $b_h \in E_B$ such that 
$c(b_h) = c$.
By Lemma \ref{lem:4.5},
we have
\begin{equation*}
B^G(b_h,b_k) = \tilde{C}(c(b_h),d(b_k)) =\tilde{C}(c,d(b_k))=1
\end{equation*}
so that
$t(b_h) =s(b_k)$,
and similarly   
$t(b_h) =s(b_l)$.
Hence  we obtain  
$s(b_k)=s(b_l)$.

(ii) is similarly shown to (i).
\end{proof}

\begin{lemma}
\label{lem:4.9}
\hspace{6cm}
\begin{enumerate}
\renewcommand{\theenumi}{\roman{enumi}}
\renewcommand{\labelenumi}{\textup{(\theenumi)}}
\item 
For $c \in \E_C$, suppose that 
\begin{align}
( \eta_{c(a_i)}\otimes_B \xi_{d(a_i)}) \otimes_A\eta_c \ne 0
& \text{ for some } a_i \in E_A, \label{eq:egei}\\
( \eta_{c(a_j)}\otimes_B \xi_{d(a_j)}) \otimes_A\eta_c \ne 0
& \text{ for some } a_j \in E_A. \label{eq:egej}
\end{align}
Then we have $t(a_i) = t(a_j)$ in $V_A$.
\item 
For $d \in \E_D$, suppose that 
\begin{align}
(\xi_{d(b_k)}\otimes_A\eta_{c(b_k)})\otimes_B
\xi_d  \ne 0
& \text{ for some } b_k \in E_B, \label{eq:gegk} \\
(\xi_{d(b_l)}\otimes_A\eta_{c(b_l)})\otimes_B
\xi_d  \ne 0
& \text{ for some } b_l \in E_B. \label{eq:gegl}
 \end{align}
Then we have $t(b_k) = t(b_l)$ in $V_B$.
\end{enumerate}
\end{lemma}
\begin{proof}
(i)
Assume that the equalities \eqref{eq:egei} and \eqref{eq:egej} hold.
By Lemma \ref{lem:4.7},
we have 
$\tilde{D}(d(a_i),c) = \tilde{D}(d(a_j),c) =1.$
Since
$\varphi_A: S_a \in \OA \longrightarrow   \eta_{c(a)}\otimes_B\xi_{d(a)} 
\in {}_A\!X_B\otimes_B {}_B\!Y_A$
yields a bimodule isomorphism, we know that 
$c : E_A \longrightarrow \E_D$ is surjective, 
so that 
there exists $a_k \in E_A$ such that 
$c(a_k) = c$.
By Lemma \ref{lem:4.5},
we have
\begin{equation*}
A^G(a_i,a_k) = \tilde{D}(d(a_i),c(a_k)) =\tilde{C}(d(a_i),c) =1
\end{equation*}
so that
$t(a_i) =s(a_k)$,
and similarly  
$t(a_j) =s(a_k)$.
Hence  we obtain 
$t(a_i) =t(a_j)$.

(ii) is similarly shown to (i).
\end{proof}

Recall that $A$ is an $N\times N$ matrix and $B$ is an $M \times M$ matrix respectively,
and $G_A = (V_A, E_A)$ and $G_B=(V_B, E_B)$ are the associated directed graphs 
such that $|V_A| = N$ and $|V_B| = M$, respectively.
We will next construct 
an $N\times M$ nonnegative matrix $C$ and
an $M \times N$ nonnegative matrix $D$ in the following way.
Take arbitrary fixed 
$u \in V_A$ and $v \in V_B$, 
define the
$(u,v)$-component $C(u,v)$ of the matrix $C$ by setting
\begin{align*}
C(u,v)
=& |
\{ c(a_i) \in \E_C \mid  \text{there exist } a_i, a_j \in E_A \text{ and } b_k \in E_B
\text{ such that } \\
 & \hspace{1cm} u=s(a_i),\,  A^G(a_i,a_j) =1,  d(a_i) = d(b_k), c(a_j) = c(b_k), \,\, v= s(b_k) \} |
\end{align*}
as in Figure 1.
\begin{figure}[htbp]
\begin{center}
\input{winfigCuv}
\end{center}
\caption{}
\end{figure}

We similarly define 
the
$(v,u)$-component $D(v,u)$ of the matrix $D$ by setting
\begin{align*}
D(v,u)
=& |
\{ d(b_k) \in \E_D \mid  \text{there exist } b_k, b_l \in E_B \text{ and } a_j \in E_A
\text{ such that } \\
& \hspace{1cm} v=s(b_k), \, B^G(b_k,b_l) =1,  c(b_k) = c(a_j), d(b_l) = d(a_j), \,\, u= s(a_j) \} |
\end{align*}
as in Figure 2.
\begin{figure}[htbp]
\begin{center}
\input{winfigDvu}
\end{center}
\caption{}
\end{figure}
We thus have 
$N \times M$ matrix $C=[C(u,v)]_{u\in V_A,v\in V_B}$,
and 
$M \times N$ matrix $D=[D(v,u)]_{v\in V_B,u\in V_A}$.

Before reaching the theorem,
we provide another lemma.
\begin{lemma}
\label{lem:4.10}
\hspace{6cm}
\begin{enumerate}
\renewcommand{\theenumi}{\roman{enumi}}
\renewcommand{\labelenumi}{\textup{(\theenumi)}}
\item 
For $c(a) \in \E_C$ with $ a \in E_A$, 
the vertices $u \in V_A$ and $v \in V_B$ satisfying $c(a) \in C(u,v)$
are unique. 
This means that if 
$c(a) \in C(u',v')$, then $u=u'$ and $v=v'$. 
\item 
For $d(b) \in \E_D$ with $b \in E_B$, 
the vertices  $v \in V_B$ and $u \in V_A$ satisfying $d(b) \in D(v,u)$
are unique. 
This means that if 
$d(b) \in D(v',u')$, then $v=v'$ and $u=u'$. 
\end{enumerate}
\end{lemma}
\begin{proof}
(i)
For 
$c(a) \in C(u,v)$, 
take 
$b_k \in E_B$ such that 
$v=s(b_k)$ and $d(a) = d(b_k)$.
As
\begin{equation}
 \eta_{c(a)}\otimes_B (\xi_{d(b_k)} \otimes_A\eta_{c(b_k)}) \ne 0, \label{ecadbk}
\end{equation}
by Lemma \ref{lem:4.8}, 
the vertex $s(b_k)$ does not depend on the choice 
of $b_k$ satisfying \eqref{ecadbk},
so that the vertex $v$ is uniquely determined by $c(a)$.

We also know  that $u = s(a)$.
Take 
$a_i \in E_A$ such that 
\begin{equation}
( \eta_{c(a_i)}\otimes_B \xi_{d(a_i)}) \otimes_A\eta_{c(a)} \ne 0. \label{ecadai}
\end{equation}
By Lemma \ref{lem:4.9}, 
the vertex $t(a_i)$, which is $s(a) =u$ does not depend on the choice 
of $a_i$ satisfying \eqref{ecadai},
so that the vertex $u$ is uniquely determined by $c(a)$.

(ii) is smilarly shown to (i).
 \end{proof}

\begin{theorem}\label{thm:main2}
$A =CD$ and $B = DC$.
\end{theorem}
\begin{proof}
We will prove $A = CD$.
Take arbitrary vertices $u_1, u_2 \in V_A$.
We note that
\begin{equation*}
A(u_1,u_2) = |\{a_i \in E_A \mid u_1 = s(a_i), \, u_2 = t(a_i) \}|.
\end{equation*} 
We will define a correspondence 
\begin{equation}
a_i \in A(u_1, u_2) \longrightarrow c(a_i)\times d(a_i) \in C(u_1,v) \times D(v,u_2) 
\text{ for some } v \in V_B.  
\end{equation}
Take an arbitrary edge $a_i \in E_A$ with
$u_1 = s(a_i), u_2 = t(a_i)$.
One may find an edge $a_j \in E_A$ such that 
$t(a_i) = s(a_j)$, and hence
$A^G(a_i, a_j) =1$.
By Lemma \ref{lem:4.5},
we have $\tilde{D}(d(a_i), c(a_j)) = 1$.
By Lemma \ref{lem:4.6},
there exists a unique edge $b_k \in E_B$ 
such that 
$d(a_i) = d(b_k), c(a_j) = c(b_k)$.
We put
$v = s(b_k) \in V_B$ so that 
$c(a_i) \in C(u_1,v)$.
The vertex $v$ is uniquely determined by $c(a_i)$
from Lemma \ref{lem:4.10}.

One may further 
find an edge $a_k \in E_A$ such that 
$t(a_j) = s(a_k)$, and hence
$A^G(a_j, a_k) =1$.
By Lemma \ref{lem:4.5},
we have $\tilde{D}(d(a_j), c(a_k)) = 1$.
By Lemma \ref{lem:4.6},
there exists an edge $b_l \in E_B$ 
such that 
$d(a_j) = d(b_l), c(a_k) = c(b_l)$,
so that 
$d(a_i)\in D(v,u_2)$.
The vertex $v$ is uniquely determined by $d(a_i)$
from Lemma \ref{lem:4.10}.
The situation is figured in Figure 3.
\begin{figure}[htbp]
\begin{center}
\input{winfigAu1u2}
\end{center}
\caption{}
\end{figure}

Conversely,
take an arbitrary fixed vertex $v \in V_B$.
Let $c(a) \in C(u_1,v), d(a')\in D(v,u_2)$.
By the condition
$d(a')\in D(v,u_2)$,
one may find
$b_k, b_l \in E_B, a_j \in E_A$ 
such that 
$B^G(b_k,b_l) =1$
and
$$
d(a') = d(b_k), \, v=s(b_k), \,  u_2 = s(a_j), \,  c(b_k) = c(a_j), \,  d(b_l) = d(a_j)
$$
as  in Figure 4.
\begin{figure}[htbp]
\begin{center}
\input{winfigDvu2}
\end{center}
\caption{}
\end{figure}
By the condition
$c(a)\in C(u_1,v)$,
one may find
$a_i', a_j'  \in E_A, b_k' \in E_B$ 
such that 
$A^G(a_i', a_j') =1$
and
$$
c(a) = c(a_i'), \, u_1=s(a_i'), \,  v = s(b_k'), \,  
d(a_i') = d(b_k'), \,  c(a_j') = c(b_k')
$$
as in Figure 5.
\begin{figure}[htbp]
\begin{center}
\input{winfigCu1v}
\end{center}
\caption{}
\end{figure}
There exists $a_h' \in E_A$ such that 
$A^G(a_h', a_i') = 1$.
By Lemma \ref{lem:4.5},
we have $\tilde{D}(d(a_h'), c(a_i')) = 1$.
By Lemma \ref{lem:4.6},
there exists an edge $b_h' \in E_B$ 
such that 
$d(b_h') = d(a_h'), c(b_h') = c(a_i')$
as in Figure 6.
\begin{figure}[htbp]
\begin{center}
\input{winfigCu1vu2}
\end{center}
\caption{}
\end{figure}

As $t(b_h') = s(b_k)$,
we have
$B^G(b_h', b_k) =1.$ 
By Lemma \ref{lem:4.5},
we have $\tilde{C}(c(b_h'), d(b_k)) = 1$.
By Lemma \ref{lem:4.6},
there exists an edge $a_i \in E_A$ 
such that 
$c(b_h') = c(a_i), d(b_k) = d(a_i)$
so that 
$c(a) = c(a_i), d(a') = d(a_i).$
Therefore we have a bijective correspondence
\begin{equation}
a_i \in A(u_1, u_2) \longleftrightarrow 
c(a_i)\times d(a_i) \in \cup_{v \in E_B}C(u_1,v) \times D(v,u_2).  
\end{equation}
This implies that 
$A(u_1,u_2) = \sum_{v \in E_B}C(u_1,v)D(v,u_2)$
and hence $A=CD$.
We may similarly prove that $B = DC$. 
\end{proof}
\begin{remark}
Let $C, D$ be the matrices in the above proof.
Put the block matrix 
$
Z =
\begin{bmatrix}
0 & C \\
D & 0
\end{bmatrix}
$
with size $N+M$.
Then the associated matrix $Z^G$ satisfies
$Z^G(c,d) = \tilde{C}(c,d)$ 
and
$Z^G(d,c) = \tilde{D}(d,c)$
for
$c\in \E_C, d \in \E_C$,
so that 
$\tilde{C} =C^G, \tilde{D} = D^G$.
We also have  identifications between $\E_C$ and $E_C$
and similarly betqeen $\E_D $ and $E_D$.  
\end{remark}
Therefore we reach the main result of the paper.
\begin{theorem}\label{thm:main3}
Let $A, B$ be nonnegative square irreducible matrices
that are not any permutations.
Then 
$A$ and $B$ are elementary equivalent, that is,
$A = CD, B = DC$ for some nonnegative rectangular matrices $C, D$,
if and only if
the Cuntz--Krieger algebras 
$(\OA, \{S_a\}_{a \in E_A})$ and $(\OB, \{ S_b\}_{b\in E_B})$ 
 with the canonical right finite bases 
are basis relatedly elementary Morita equivalent.
\end{theorem}
As a corollary we have 
\begin{corollary}\label{cor:main4}
Let $A, B$ be nonnegative square irreducible matrices
that are not any permutations.
The two-sided topological Markov shifts 
$(\bar{X}_A, \bar{\sigma}_A)$ and  
$(\bar{X}_B, \bar{\sigma}_B)$
are topologically conjugate
if and only if
the Cuntz--Krieger algebras 
$(\OA, \{S_a\}_{a \in E_A})$ and $(\OB, \{ S_b\}_{b\in E_B})$ 
with the canonical right finite bases 
are basis relatedly Morita equivalent.
\end{corollary}


\medskip

{\it Acknowledgments:}
This work was also supported by JSPS KAKENHI Grant Number 15K04896.


\end{document}

%% file: winfigCuv.tex
\unitlength 0.1in
\begin{picture}( 48.0200, 13.7900)(  4.9800,-31.6900)
%
\special{pn 8}%
\special{ar 644 3058 106 112  2.1890927 6.2831853}%
\special{ar 644 3058 106 112  0.0000000 2.1614167}%
%
\special{pn 8}%
\special{ar 1616 2020 106 112  2.1774014 6.2831853}%
\special{ar 1616 2020 106 112  0.0000000 2.1535358}%
%
\special{pn 8}%
\special{ar 2588 3058 106 112  2.1822139 6.2831853}%
\special{ar 2588 3058 106 112  0.0000000 2.1535358}%
%
\special{pn 8}%
\special{ar 4532 3058 106 112  2.1822139 6.2831853}%
\special{ar 4532 3058 106 112  0.0000000 2.1535358}%
%
\special{pn 8}%
\special{ar 3560 2020 106 112  2.1842564 6.2831853}%
\special{ar 3560 2020 106 112  0.0000000 2.1614167}%
%
\special{pn 8}%
\special{pa 3658 2124}%
\special{pa 4436 2954}%
\special{fp}%
\special{sh 1}%
\special{pa 4436 2954}%
\special{pa 4404 2892}%
\special{pa 4400 2916}%
\special{pa 4376 2920}%
\special{pa 4436 2954}%
\special{fp}%
%
\special{pn 8}%
\special{pa 1712 2114}%
\special{pa 2490 2944}%
\special{fp}%
\special{sh 1}%
\special{pa 2490 2944}%
\special{pa 2460 2882}%
\special{pa 2454 2904}%
\special{pa 2430 2908}%
\special{pa 2490 2944}%
\special{fp}%
%
\special{pn 8}%
\special{pa 740 2944}%
\special{pa 1518 2114}%
\special{fp}%
\special{sh 1}%
\special{pa 1518 2114}%
\special{pa 1458 2150}%
\special{pa 1482 2154}%
\special{pa 1488 2176}%
\special{pa 1518 2114}%
\special{fp}%
%
\special{pn 8}%
\special{pa 2694 2944}%
\special{pa 3472 2114}%
\special{fp}%
\special{sh 1}%
\special{pa 3472 2114}%
\special{pa 3412 2150}%
\special{pa 3436 2154}%
\special{pa 3442 2176}%
\special{pa 3472 2114}%
\special{fp}%
\put(19.5500,-26.9400){\makebox(0,0){$d(a_i)$}}%
\put(30.8300,-22.3800){\makebox(0,0){$c(b_k)$}}%
\put(31.4000,-27.0000){\makebox(0,0){$c(a_j)$}}%
\put(11.9000,-27.0000){\makebox(0,0){$c(a_i)$}}%
%
\special{pn 8}%
\special{pa 818 3078}%
\special{pa 2432 3078}%
\special{fp}%
\special{sh 1}%
\special{pa 2432 3078}%
\special{pa 2366 3058}%
\special{pa 2380 3078}%
\special{pa 2366 3098}%
\special{pa 2432 3078}%
\special{fp}%
%
\special{pn 8}%
\special{pa 1762 2010}%
\special{pa 3376 2010}%
\special{fp}%
\special{sh 1}%
\special{pa 3376 2010}%
\special{pa 3308 1990}%
\special{pa 3322 2010}%
\special{pa 3308 2030}%
\special{pa 3376 2010}%
\special{fp}%
%
\special{pn 8}%
\special{pa 3686 2010}%
\special{pa 5300 2010}%
\special{fp}%
\special{sh 1}%
\special{pa 5300 2010}%
\special{pa 5234 1990}%
\special{pa 5248 2010}%
\special{pa 5234 2030}%
\special{pa 5300 2010}%
\special{fp}%
%
\special{pn 8}%
\special{pa 2752 3078}%
\special{pa 4368 3078}%
\special{fp}%
\special{sh 1}%
\special{pa 4368 3078}%
\special{pa 4300 3058}%
\special{pa 4314 3078}%
\special{pa 4300 3098}%
\special{pa 4368 3078}%
\special{fp}%
\put(15.6600,-32.4400){\makebox(0,0){$a_i$}}%
\put(25.0000,-18.7500){\makebox(0,0){$b_k$}}%
\put(34.8200,-32.4400){\makebox(0,0){$a_j$}}%
\put(15.9500,-20.3100){\makebox(0,0){$v$}}%
\put(6.3300,-30.7800){\makebox(0,0){$u$}}%
\put(20.4300,-22.2800){\makebox(0,0){$d(b_k)$}}%
\end{picture}%

%% file: winfigDvu.tex
\unitlength 0.1in
\begin{picture}( 48.4500, 13.7900)( 11.0000,-31.5900)
%
\special{pn 8}%
\special{ar 5080 2010 106 112  2.1842564 6.2831853}%
\special{ar 5080 2010 106 112  0.0000000 2.1614167}%
%
\special{pn 8}%
\special{ar 1256 2010 106 112  2.1774014 6.2831853}%
\special{ar 1256 2010 106 112  0.0000000 2.1535358}%
%
\special{pn 8}%
\special{ar 2228 3048 106 112  2.1822139 6.2831853}%
\special{ar 2228 3048 106 112  0.0000000 2.1535358}%
%
\special{pn 8}%
\special{ar 4172 3048 106 112  2.1822139 6.2831853}%
\special{ar 4172 3048 106 112  0.0000000 2.1535358}%
%
\special{pn 8}%
\special{ar 3200 2010 106 112  2.1842564 6.2831853}%
\special{ar 3200 2010 106 112  0.0000000 2.1614167}%
%
\special{pn 8}%
\special{pa 3298 2114}%
\special{pa 4076 2944}%
\special{fp}%
\special{sh 1}%
\special{pa 4076 2944}%
\special{pa 4044 2882}%
\special{pa 4040 2906}%
\special{pa 4016 2910}%
\special{pa 4076 2944}%
\special{fp}%
%
\special{pn 8}%
\special{pa 1352 2104}%
\special{pa 2130 2934}%
\special{fp}%
\special{sh 1}%
\special{pa 2130 2934}%
\special{pa 2100 2872}%
\special{pa 2094 2894}%
\special{pa 2070 2898}%
\special{pa 2130 2934}%
\special{fp}%
%
\special{pn 8}%
\special{pa 4270 2960}%
\special{pa 5048 2132}%
\special{fp}%
\special{sh 1}%
\special{pa 5048 2132}%
\special{pa 4988 2166}%
\special{pa 5012 2170}%
\special{pa 5018 2194}%
\special{pa 5048 2132}%
\special{fp}%
%
\special{pn 8}%
\special{pa 2334 2934}%
\special{pa 3112 2104}%
\special{fp}%
\special{sh 1}%
\special{pa 3112 2104}%
\special{pa 3052 2140}%
\special{pa 3076 2144}%
\special{pa 3082 2166}%
\special{pa 3112 2104}%
\special{fp}%
\put(27.2300,-22.2800){\makebox(0,0){$c(b_k)$}}%
\put(28.5900,-26.7000){\makebox(0,0){$c(a_j)$}}%
\put(35.7900,-26.8400){\makebox(0,0){$d(a_j)$}}%
%
\special{pn 8}%
\special{pa 1402 2000}%
\special{pa 3016 2000}%
\special{fp}%
\special{sh 1}%
\special{pa 3016 2000}%
\special{pa 2948 1980}%
\special{pa 2962 2000}%
\special{pa 2948 2020}%
\special{pa 3016 2000}%
\special{fp}%
%
\special{pn 8}%
\special{pa 3326 2000}%
\special{pa 4940 2000}%
\special{fp}%
\special{sh 1}%
\special{pa 4940 2000}%
\special{pa 4874 1980}%
\special{pa 4888 2000}%
\special{pa 4874 2020}%
\special{pa 4940 2000}%
\special{fp}%
%
\special{pn 8}%
\special{pa 2392 3068}%
\special{pa 4008 3068}%
\special{fp}%
\special{sh 1}%
\special{pa 4008 3068}%
\special{pa 3940 3048}%
\special{pa 3954 3068}%
\special{pa 3940 3088}%
\special{pa 4008 3068}%
\special{fp}%
\put(21.4000,-18.6500){\makebox(0,0){$b_k$}}%
\put(31.2200,-32.3400){\makebox(0,0){$a_j$}}%
\put(12.3500,-20.2100){\makebox(0,0){$v$}}%
\put(22.3000,-30.6000){\makebox(0,0){$u$}}%
\put(16.8300,-22.1800){\makebox(0,0){$d(b_k)$}}%
\put(41.0000,-18.8000){\makebox(0,0){$b_l$}}%
\put(36.2000,-22.2000){\makebox(0,0){$d(b_l)$}}%
%
\special{pn 8}%
\special{pa 4330 3060}%
\special{pa 5946 3060}%
\special{fp}%
\special{sh 1}%
\special{pa 5946 3060}%
\special{pa 5878 3040}%
\special{pa 5892 3060}%
\special{pa 5878 3080}%
\special{pa 5946 3060}%
\special{fp}%
\end{picture}%

%% file: winfigAu1u2.tex
\unitlength 0.1in
\begin{picture}( 61.4000, 13.7900)(  4.1500,-31.6900)
%
\special{pn 8}%
\special{ar 644 3058 106 112  2.1890927 6.2831853}%
\special{ar 644 3058 106 112  0.0000000 2.1614167}%
%
\special{pn 8}%
\special{ar 1616 2020 106 112  2.1774014 6.2831853}%
\special{ar 1616 2020 106 112  0.0000000 2.1535358}%
%
\special{pn 8}%
\special{ar 2588 3058 106 112  2.1822139 6.2831853}%
\special{ar 2588 3058 106 112  0.0000000 2.1535358}%
%
\special{pn 8}%
\special{ar 4532 3058 106 112  2.1822139 6.2831853}%
\special{ar 4532 3058 106 112  0.0000000 2.1535358}%
%
\special{pn 8}%
\special{ar 3560 2020 106 112  2.1842564 6.2831853}%
\special{ar 3560 2020 106 112  0.0000000 2.1614167}%
%
\special{pn 8}%
\special{pa 3658 2124}%
\special{pa 4436 2954}%
\special{fp}%
\special{sh 1}%
\special{pa 4436 2954}%
\special{pa 4404 2892}%
\special{pa 4400 2916}%
\special{pa 4376 2920}%
\special{pa 4436 2954}%
\special{fp}%
%
\special{pn 8}%
\special{pa 1712 2114}%
\special{pa 2490 2944}%
\special{fp}%
\special{sh 1}%
\special{pa 2490 2944}%
\special{pa 2460 2882}%
\special{pa 2454 2904}%
\special{pa 2430 2908}%
\special{pa 2490 2944}%
\special{fp}%
%
\special{pn 8}%
\special{pa 740 2944}%
\special{pa 1518 2114}%
\special{fp}%
\special{sh 1}%
\special{pa 1518 2114}%
\special{pa 1458 2150}%
\special{pa 1482 2154}%
\special{pa 1488 2176}%
\special{pa 1518 2114}%
\special{fp}%
%
\special{pn 8}%
\special{pa 2694 2944}%
\special{pa 3472 2114}%
\special{fp}%
\special{sh 1}%
\special{pa 3472 2114}%
\special{pa 3412 2150}%
\special{pa 3436 2154}%
\special{pa 3442 2176}%
\special{pa 3472 2114}%
\special{fp}%
\put(19.5500,-26.9400){\makebox(0,0){$d(a_i)$}}%
\put(30.8300,-22.3800){\makebox(0,0){$c(b_k)$}}%
\put(31.4000,-27.0000){\makebox(0,0){$c(a_j)$}}%
\put(11.9000,-27.0000){\makebox(0,0){$c(a_i)$}}%
\put(39.3900,-26.9400){\makebox(0,0){$d(a_j)$}}%
%
\special{pn 8}%
\special{pa 818 3078}%
\special{pa 2432 3078}%
\special{fp}%
\special{sh 1}%
\special{pa 2432 3078}%
\special{pa 2366 3058}%
\special{pa 2380 3078}%
\special{pa 2366 3098}%
\special{pa 2432 3078}%
\special{fp}%
%
\special{pn 8}%
\special{pa 1762 2010}%
\special{pa 3376 2010}%
\special{fp}%
\special{sh 1}%
\special{pa 3376 2010}%
\special{pa 3308 1990}%
\special{pa 3322 2010}%
\special{pa 3308 2030}%
\special{pa 3376 2010}%
\special{fp}%
%
\special{pn 8}%
\special{pa 3686 2010}%
\special{pa 5300 2010}%
\special{fp}%
\special{sh 1}%
\special{pa 5300 2010}%
\special{pa 5234 1990}%
\special{pa 5248 2010}%
\special{pa 5234 2030}%
\special{pa 5300 2010}%
\special{fp}%
%
\special{pn 8}%
\special{pa 2752 3078}%
\special{pa 4368 3078}%
\special{fp}%
\special{sh 1}%
\special{pa 4368 3078}%
\special{pa 4300 3058}%
\special{pa 4314 3078}%
\special{pa 4300 3098}%
\special{pa 4368 3078}%
\special{fp}%
\put(15.6600,-32.4400){\makebox(0,0){$a_i$}}%
\put(25.0000,-18.7500){\makebox(0,0){$b_k$}}%
\put(34.8200,-32.4400){\makebox(0,0){$a_j$}}%
\put(15.9500,-20.3100){\makebox(0,0){$v$}}%
\put(20.4300,-22.2800){\makebox(0,0){$d(b_k)$}}%
\put(25.8000,-30.8000){\makebox(0,0){$u_2$}}%
\put(6.4000,-30.8000){\makebox(0,0){$u_1$}}%
\put(43.2000,-18.9000){\makebox(0,0){$b_l$}}%
%
\special{pn 8}%
\special{pa 4600 2960}%
\special{pa 5378 2130}%
\special{fp}%
\special{sh 1}%
\special{pa 5378 2130}%
\special{pa 5318 2166}%
\special{pa 5342 2170}%
\special{pa 5348 2192}%
\special{pa 5378 2130}%
\special{fp}%
%
\special{pn 8}%
\special{ar 5450 2020 106 112  2.1842564 6.2831853}%
\special{ar 5450 2020 106 112  0.0000000 2.1614167}%
\put(39.9000,-22.6000){\makebox(0,0){$d(b_l)$}}%
%
\special{pn 8}%
\special{ar 6450 3052 106 112  2.1774014 6.2831853}%
\special{ar 6450 3052 106 112  0.0000000 2.1535358}%
%
\special{pn 8}%
\special{pa 4670 3074}%
\special{pa 6286 3074}%
\special{fp}%
\special{sh 1}%
\special{pa 6286 3074}%
\special{pa 6218 3054}%
\special{pa 6232 3074}%
\special{pa 6218 3094}%
\special{pa 6286 3074}%
\special{fp}%
\put(52.9000,-32.3000){\makebox(0,0){$a_k$}}%
\put(50.2000,-27.2000){\makebox(0,0){$c(a_k)$}}%
\put(49.2000,-22.5000){\makebox(0,0){$c(b_l)$}}%
\end{picture}%

%% file: winfigDvu2.tex
\unitlength 0.1in
\begin{picture}( 40.8500, 13.7900)( 11.0000,-31.5900)
%
\special{pn 8}%
\special{ar 5080 2010 106 112  2.1842564 6.2831853}%
\special{ar 5080 2010 106 112  0.0000000 2.1614167}%
%
\special{pn 8}%
\special{ar 1256 2010 106 112  2.1774014 6.2831853}%
\special{ar 1256 2010 106 112  0.0000000 2.1535358}%
%
\special{pn 8}%
\special{ar 2228 3048 106 112  2.1822139 6.2831853}%
\special{ar 2228 3048 106 112  0.0000000 2.1535358}%
%
\special{pn 8}%
\special{ar 4172 3048 106 112  2.1822139 6.2831853}%
\special{ar 4172 3048 106 112  0.0000000 2.1535358}%
%
\special{pn 8}%
\special{ar 3200 2010 106 112  2.1842564 6.2831853}%
\special{ar 3200 2010 106 112  0.0000000 2.1614167}%
%
\special{pn 8}%
\special{pa 3298 2114}%
\special{pa 4076 2944}%
\special{fp}%
\special{sh 1}%
\special{pa 4076 2944}%
\special{pa 4044 2882}%
\special{pa 4040 2906}%
\special{pa 4016 2910}%
\special{pa 4076 2944}%
\special{fp}%
%
\special{pn 8}%
\special{pa 1352 2104}%
\special{pa 2130 2934}%
\special{fp}%
\special{sh 1}%
\special{pa 2130 2934}%
\special{pa 2100 2872}%
\special{pa 2094 2894}%
\special{pa 2070 2898}%
\special{pa 2130 2934}%
\special{fp}%
%
\special{pn 8}%
\special{pa 4270 2960}%
\special{pa 5048 2132}%
\special{fp}%
\special{sh 1}%
\special{pa 5048 2132}%
\special{pa 4988 2166}%
\special{pa 5012 2170}%
\special{pa 5018 2194}%
\special{pa 5048 2132}%
\special{fp}%
%
\special{pn 8}%
\special{pa 2334 2934}%
\special{pa 3112 2104}%
\special{fp}%
\special{sh 1}%
\special{pa 3112 2104}%
\special{pa 3052 2140}%
\special{pa 3076 2144}%
\special{pa 3082 2166}%
\special{pa 3112 2104}%
\special{fp}%
\put(27.2300,-22.2800){\makebox(0,0){$c(b_k)$}}%
\put(28.1000,-26.7000){\makebox(0,0){$c(a_j)$}}%
\put(35.7900,-26.8400){\makebox(0,0){$d(a_j)$}}%
%
\special{pn 8}%
\special{pa 1402 2000}%
\special{pa 3016 2000}%
\special{fp}%
\special{sh 1}%
\special{pa 3016 2000}%
\special{pa 2948 1980}%
\special{pa 2962 2000}%
\special{pa 2948 2020}%
\special{pa 3016 2000}%
\special{fp}%
%
\special{pn 8}%
\special{pa 3326 2000}%
\special{pa 4940 2000}%
\special{fp}%
\special{sh 1}%
\special{pa 4940 2000}%
\special{pa 4874 1980}%
\special{pa 4888 2000}%
\special{pa 4874 2020}%
\special{pa 4940 2000}%
\special{fp}%
%
\special{pn 8}%
\special{pa 2392 3068}%
\special{pa 4008 3068}%
\special{fp}%
\special{sh 1}%
\special{pa 4008 3068}%
\special{pa 3940 3048}%
\special{pa 3954 3068}%
\special{pa 3940 3088}%
\special{pa 4008 3068}%
\special{fp}%
\put(21.4000,-18.6500){\makebox(0,0){$b_k$}}%
\put(31.2200,-32.3400){\makebox(0,0){$a_j$}}%
\put(12.3500,-20.2100){\makebox(0,0){$v$}}%
\put(16.8300,-22.1800){\makebox(0,0){$d(b_k)$}}%
\put(41.0000,-18.8000){\makebox(0,0){$b_l$}}%
\put(36.2000,-22.2000){\makebox(0,0){$d(b_l)$}}%
\put(22.3000,-30.6000){\makebox(0,0){$u_2$}}%
\put(16.0000,-26.8000){\makebox(0,0){$d(a')$}}%
\end{picture}%

%% file: winfigCu1v.tex
\unitlength 0.1in
\begin{picture}( 42.3700, 13.8400)(  4.0000,-31.6900)
%
\special{pn 8}%
\special{ar 644 3058 106 112  2.1890927 6.2831853}%
\special{ar 644 3058 106 112  0.0000000 2.1614167}%
%
\special{pn 8}%
\special{ar 1616 2020 106 112  2.1774014 6.2831853}%
\special{ar 1616 2020 106 112  0.0000000 2.1535358}%
%
\special{pn 8}%
\special{ar 2588 3058 106 112  2.1822139 6.2831853}%
\special{ar 2588 3058 106 112  0.0000000 2.1535358}%
%
\special{pn 8}%
\special{ar 4532 3058 106 112  2.1822139 6.2831853}%
\special{ar 4532 3058 106 112  0.0000000 2.1535358}%
%
\special{pn 8}%
\special{ar 3560 2020 106 112  2.1842564 6.2831853}%
\special{ar 3560 2020 106 112  0.0000000 2.1614167}%
%
\special{pn 8}%
\special{pa 3658 2124}%
\special{pa 4436 2954}%
\special{fp}%
\special{sh 1}%
\special{pa 4436 2954}%
\special{pa 4404 2892}%
\special{pa 4400 2916}%
\special{pa 4376 2920}%
\special{pa 4436 2954}%
\special{fp}%
%
\special{pn 8}%
\special{pa 1712 2114}%
\special{pa 2490 2944}%
\special{fp}%
\special{sh 1}%
\special{pa 2490 2944}%
\special{pa 2460 2882}%
\special{pa 2454 2904}%
\special{pa 2430 2908}%
\special{pa 2490 2944}%
\special{fp}%
%
\special{pn 8}%
\special{pa 740 2944}%
\special{pa 1518 2114}%
\special{fp}%
\special{sh 1}%
\special{pa 1518 2114}%
\special{pa 1458 2150}%
\special{pa 1482 2154}%
\special{pa 1488 2176}%
\special{pa 1518 2114}%
\special{fp}%
%
\special{pn 8}%
\special{pa 2694 2944}%
\special{pa 3472 2114}%
\special{fp}%
\special{sh 1}%
\special{pa 3472 2114}%
\special{pa 3412 2150}%
\special{pa 3436 2154}%
\special{pa 3442 2176}%
\special{pa 3472 2114}%
\special{fp}%
%
\special{pn 8}%
\special{pa 818 3078}%
\special{pa 2432 3078}%
\special{fp}%
\special{sh 1}%
\special{pa 2432 3078}%
\special{pa 2366 3058}%
\special{pa 2380 3078}%
\special{pa 2366 3098}%
\special{pa 2432 3078}%
\special{fp}%
%
\special{pn 8}%
\special{pa 1762 2010}%
\special{pa 3376 2010}%
\special{fp}%
\special{sh 1}%
\special{pa 3376 2010}%
\special{pa 3308 1990}%
\special{pa 3322 2010}%
\special{pa 3308 2030}%
\special{pa 3376 2010}%
\special{fp}%
%
\special{pn 8}%
\special{pa 2752 3078}%
\special{pa 4368 3078}%
\special{fp}%
\special{sh 1}%
\special{pa 4368 3078}%
\special{pa 4300 3058}%
\special{pa 4314 3078}%
\special{pa 4300 3098}%
\special{pa 4368 3078}%
\special{fp}%
\put(15.9500,-20.3100){\makebox(0,0){$v$}}%
\put(24.6000,-18.7000){\makebox(0,0){$b'_k$}}%
\put(20.5000,-22.4000){\makebox(0,0){$d(b'_k)$}}%
\put(30.5000,-22.3000){\makebox(0,0){$c(b'_k)$}}%
\put(14.5000,-32.3000){\makebox(0,0){$a'_i$}}%
\put(11.9000,-27.1000){\makebox(0,0){$c(a'_i)$}}%
\put(19.9000,-27.1000){\makebox(0,0){$d(a'_i)$}}%
\put(34.6000,-32.2000){\makebox(0,0){$a'_j$}}%
\put(31.2000,-27.2000){\makebox(0,0){$c(a'_j)$}}%
\put(39.4000,-27.2000){\makebox(0,0){$d(a'_j)$}}%
\put(6.4000,-30.7000){\makebox(0,0){$u_1$}}%
\put(25.9000,-30.8000){\makebox(0,0){$u_2$}}%
\put(6.7000,-27.3000){\makebox(0,0){$c(a)$}}%
\end{picture}%

%% file: winfigCu1vu2.tex
\unitlength 0.1in
\begin{picture}( 40.7000, 13.8000)(  6.9500,-31.9000)
%
\special{pn 8}%
\special{ar 800 2030 106 112  2.1842564 6.2831853}%
\special{ar 800 2030 106 112  0.0000000 2.1614167}%
%
\special{pn 8}%
\special{ar 2716 2040 106 112  2.1774014 6.2831853}%
\special{ar 2716 2040 106 112  0.0000000 2.1535358}%
%
\special{pn 8}%
\special{ar 3688 3078 106 112  2.1822139 6.2831853}%
\special{ar 3688 3078 106 112  0.0000000 2.1535358}%
%
\special{pn 8}%
\special{ar 1714 3078 106 112  2.1774014 6.2831853}%
\special{ar 1714 3078 106 112  0.0000000 2.1535358}%
%
\special{pn 8}%
\special{ar 4660 2040 106 112  2.1842564 6.2831853}%
\special{ar 4660 2040 106 112  0.0000000 2.1614167}%
%
\special{pn 8}%
\special{pa 878 2130}%
\special{pa 1656 2960}%
\special{fp}%
\special{sh 1}%
\special{pa 1656 2960}%
\special{pa 1626 2898}%
\special{pa 1620 2922}%
\special{pa 1596 2926}%
\special{pa 1656 2960}%
\special{fp}%
%
\special{pn 8}%
\special{pa 2812 2134}%
\special{pa 3590 2964}%
\special{fp}%
\special{sh 1}%
\special{pa 3590 2964}%
\special{pa 3560 2902}%
\special{pa 3554 2924}%
\special{pa 3530 2928}%
\special{pa 3590 2964}%
\special{fp}%
%
\special{pn 8}%
\special{pa 1852 2976}%
\special{pa 2630 2148}%
\special{fp}%
\special{sh 1}%
\special{pa 2630 2148}%
\special{pa 2570 2182}%
\special{pa 2594 2186}%
\special{pa 2598 2210}%
\special{pa 2630 2148}%
\special{fp}%
%
\special{pn 8}%
\special{pa 3794 2964}%
\special{pa 4572 2134}%
\special{fp}%
\special{sh 1}%
\special{pa 4572 2134}%
\special{pa 4512 2170}%
\special{pa 4536 2174}%
\special{pa 4542 2196}%
\special{pa 4572 2134}%
\special{fp}%
\put(41.8300,-22.5800){\makebox(0,0){$c(b_k)$}}%
%
\special{pn 8}%
\special{pa 2862 2030}%
\special{pa 4476 2030}%
\special{fp}%
\special{sh 1}%
\special{pa 4476 2030}%
\special{pa 4408 2010}%
\special{pa 4422 2030}%
\special{pa 4408 2050}%
\special{pa 4476 2030}%
\special{fp}%
%
\special{pn 8}%
\special{pa 950 2030}%
\special{pa 2564 2030}%
\special{fp}%
\special{sh 1}%
\special{pa 2564 2030}%
\special{pa 2498 2010}%
\special{pa 2512 2030}%
\special{pa 2498 2050}%
\special{pa 2564 2030}%
\special{fp}%
\put(36.0000,-18.9500){\makebox(0,0){$b_k$}}%
\put(26.9500,-20.5100){\makebox(0,0){$v$}}%
\put(31.4300,-22.4800){\makebox(0,0){$d(b_k)$}}%
\put(36.9000,-30.9000){\makebox(0,0){$u_2$}}%
\put(30.6000,-27.1000){\makebox(0,0){$d(a')$}}%
\put(16.8000,-19.2000){\makebox(0,0){$b'_h$}}%
\put(22.2000,-22.6000){\makebox(0,0){$c(b'_h)$}}%
\put(12.8000,-22.7000){\makebox(0,0){$d(b'_h)$}}%
\put(17.0000,-31.0000){\makebox(0,0){$u_1$}}%
\put(22.8000,-27.3000){\makebox(0,0){$c(a)$}}%
\end{picture}%